\documentclass[english]{article}
\usepackage[T1]{fontenc}
\usepackage[latin1]{inputenc}
\usepackage{geometry}
\usepackage{float}
\usepackage{mathrsfs}
\usepackage{amsmath}
\usepackage{amssymb}
\usepackage{graphicx}
\usepackage{subfigure}

\makeatletter

\floatstyle{ruled}
\newfloat{algorithm}{tbp}{loa}
\providecommand{\algorithmname}{Algorithm}
\floatname{algorithm}{\protect\algorithmname}

\usepackage{amsthm}
\usepackage{mathrsfs}
\usepackage{amsfonts}
\usepackage{epsfig}
\usepackage{bm}
\usepackage{mathrsfs}
\usepackage{enumerate}

\usepackage{subfig}\usepackage[all]{xy}

\newtheorem{ass}{Assumption}[section]
\newtheorem{theorem}{Theorem}[section]
\newtheorem{lem}{Lemma}[section]
\newtheorem{rem}{Remark}[section]
\newtheorem{prop}{Proposition}[section]

\newcounter{hypA}
\newenvironment{hypA}{\refstepcounter{hypA}\begin{itemize}
  \item[({\bf A\arabic{hypA}})]}{\end{itemize}}

\usepackage{babel}

\date{}

\makeatother

\newcommand{\calU}{\mathcal{U}}

\newcommand{\bbE}{\mathbb{E}}

\newcommand{\bbR}{\mathbb{R}}

\begin{document}

\begin{center}

{\Large \textbf{A Multilevel Approach for Stochastic Nonlinear Optimal Control}}

\vspace{0.5cm}

BY AJAY JASRA$^{1}$, JEREMY HENG$^{2}$, YAXIAN XU$^{3}$ \& ADRIAN N. BISHOP$^{4}$

{\footnotesize $^{1,3}$Department of Statistics \& Applied Probability,
National University of Singapore, Singapore, 117546, SG.}
{\footnotesize E-Mail:\,} \texttt{\emph{\footnotesize staja@nus.edu.sg}, \emph{\footnotesize{a0078115@u.nus.edu}} }\\
{\footnotesize $^{2}$ESSEC Business School, 5 Nepal Park, Singapore 139408, SG.}
{\footnotesize E-Mail:\,} \texttt{\emph{\footnotesize heng@essec.edu}}\\
{\footnotesize $^{4}$University of Technology Sydney, AUS.}
{\footnotesize E-Mail:\,} \texttt{\emph{\footnotesize adrian.bishop@uts.edu.au}}\\
\end{center}

\begin{abstract}
We consider a class of finite time horizon nonlinear stochastic optimal control problem, where the control acts additively on the dynamics 
and the control cost is quadratic. This framework is flexible and has found applications in many domains. 
Although the optimal control admits a path integral representation for this class of control problems, 
efficient computation of the associated path integrals remains a challenging Monte Carlo task. 
The focus of this article is to propose a new Monte Carlo approach that significantly improves upon existing methodology. 
Our proposed methodology first tackles the issue of exponential growth in variance with the time horizon by casting  
optimal control estimation as a smoothing problem for a state space model associated with the control problem, 
and applying smoothing algorithms based on particle Markov chain Monte Carlo. 
To further reduce computational cost, we then develop a multilevel Monte Carlo method which allows us to obtain 
an estimator of the optimal control with $\mathcal{O}(\epsilon^2)$ mean squared error with a computational cost of
$\mathcal{O}(\epsilon^{-2}\log(\epsilon)^2)$. 
In contrast, a computational cost of $\mathcal{O}(\epsilon^{-3})$ is required for existing methodology to achieve the same mean squared error. 
Our approach is illustrated on two numerical examples, which validate our theory. 
\end{abstract}

\section{Introduction}
We consider a class of finite time horizon nonlinear stochastic optimal control problem, where the control acts additively on the dynamics 
and the control cost is quadratic \cite{kappen2005path,kappen2005linear}. This framework is flexible and has applications in domains 
such as robotics \cite{theodorou2010generalized}, epidemiology \cite{Tornatore2014,Witbooi2015}, reinforcement learning \cite{theodorou2010generalized}, and nonlinear particle smoothing \cite{kappenruiz2016,ruizkappen2017}. 
For this class of control problems, the nonlinear Hamilton-Jacobi-Bellman equation can be reduced to a linear equation by applying a suitable logarithmic transformation. Although this allows the optimal control to admit a closed form expression via the Feynman-Kac formula, 
efficient computation of the associated path integrals remains a challenging Monte Carlo task. 
Simple approaches based on simulating the uncontrolled dynamics and performing normalized importance sampling \cite{bert} 
often suffer from exponential growth in variance with the time horizon. Hence to accurately estimate the optimal control, one might require 
an exponentially commensurate number of Monte Carlo simulations. As optimal importance sampling, i.e. zero variance estimation of the optimal control, 
is achieved when one simulates from the optimally controlled dynamics \cite[Theorem 2]{thijssen2014path}, this prompts an iterative 
procedure \cite{Theodorou_Todorov_2012,thijssen2014path,menchonkappen2018} to estimate optimal control. Although these iterative importance 
control methods can often give substantial variance reduction, they typically require parameterizing the form of the control.
The focus of this article is to propose a new Monte Carlo approach to the above path integral control problem. 
Our proposed methodology first tackles the issue of exponential growth in variance with the time horizon by casting  
optimal control estimation as a smoothing problem for a state space model associated with the control problem, 
and applying state-of-the-art smoothing algorithms based on particle Markov chain Monte Carlo (MCMC) \cite{andrieu}. 

To further reduce computational cost, we then consider the multilevel Monte Carlo (MLMC) method \cite{giles,giles1,hein} 
which is particularly well-suited to the problem at hand as path integrals are expectations w.r.t.~a continuum model, defined by 
the probability law of the uncontrolled stochastic differential equation (SDE).
For numerical implementation to be tractable, one must typically resort to discretizing the continuum model 
(for instance using Euler discretizations of SDEs) and considering expectations w.r.t.~the discretized model. 
As the time discretization becomes more precise, the approximation becomes more accurate, but simultaneously, the associated cost of computation increases.
The MLMC method considers a seemingly trivial telescoping sum representation of
the expectation w.r.t.~the most precise discretization, where the summands are differences of expectations of increasingly coarsely discretized approximations. 
The key idea is to approximate the differences by sampling dependent \emph{couplings}
of the two probability measures in the difference, independently for each difference. 
In some contexts, if the couplings are appropriately constructed (see e.g.~\cite{giles}) then relative to Monte Carlo estimation for the most precise expectation, the computational cost can be significantly reduced. 
In the case of optimal control estimation, one cannot adopt the MLMC method in \cite{giles} directly, as exact sampling from smoothing 
distributions and their couplings is intractable. 
To circumvent this difficulty, we adapt the ideas in \cite{jasra} (see also \cite{franks}) to our context to develop a MLMC method that is based on
MCMC sampling methods. 
We then establish that, for some given $\epsilon>0$, for our proposed estimator of the optimal control to have $\mathcal{O}(\epsilon^2)$ 
mean squared error (MSE), the computational cost required is $\mathcal{O}(\epsilon^{-2}\log(\epsilon)^2)$. 
In contrast, a computational cost of $\mathcal{O}(\epsilon^{-3})$ is required for existing methodology to achieve the same MSE.
In the case where one assumes a parametric form of the control, we note that previous work in \cite{menchonkappen2018} 
have considered applying the standard MLMC method in \cite{giles} to reduce the computational cost within an iterative importance control 
scheme to update the control parameters. 

This article is structured as follows. In Section \ref{sec:problem}, we begin by detailing the stochastic optimal control problem of interest and its 
path integral formulation. We then describe our proposed methodology to compute the optimal control in Section \ref{sec:comp}, and 
state some theoretical results on its complexity in Section \ref{sec:theory}. 
In Section \ref{sec:numerics}, we validate our theory on two examples, including a nonlinear stochastic compartmental model for an epidemic with cost-controlled vaccination. The appendix features the assumptions and proofs for our complexity theorem in Section \ref{sec:theory}.

\section{Nonlinear stochastic optimal control}\label{sec:problem}
We consider a nonlinear controlled process $\{X_t, 0\leq t\leq T\}$ in $\mathbb{R}^n$, defined as the solution of the following SDE
\begin{equation}
	dX_t = f(X_t)dt + e(X_t)u(t,X_t)dt + g(X_{t})dW_{t} 
\label{montecarlosystemcontrol}
\end{equation}
with initial condition $X_{0} = x_{0} \in\bbR^n$. The above is to be understood in the Ito sense \cite{arnold1974stochastic} and $\{W_t, 0\leq t\leq T\}$ is a standard Brownian motion in $\mathbb{R}^d$. We assume that $f:\bbR^n\to\bbR^n$, $e:\bbR^n\to\bbR^{n\times m}$ and $g:\bbR^n\to\bbR^{n\times d}$ are twice differentiable and there exists a constant $c_1>0$ such that 
$$
	|f(x)-f(y)|+|g(x)-g(y)| + |e(x)-e(y)| \leq c_1|x-y|, 
$$
for all $(x,y)\in\bbR^n\times\bbR^n$. Without any loss of generality, we suppose that $e$ and $g$ (which may be non-square) have full rank. Note that 
the latter implies existence and uniqueness of left-inverses, i.e. functions $e^{-1} : \bbR^n\rightarrow \bbR^{m\times n}$ and 
$g^{-1} : \bbR^n\rightarrow \bbR^{d\times n}$ such that $(e^{-1}e)(x)=I_m$ and $(g^{-1}g)(x)=I_d$ for all $x \in\bbR^n$. 
Lastly, we assume that $g(x)g(x)^{\top}$ is uniformly positive definite over $x\in\mathbb{R}^n$.
For a fixed time interval $[t_0,t_1]\subseteq[0,T]$, the set of admissible controls $\calU_{[t_0,t_1]}$ we shall consider are 
Borel measurable functions $u:[t_0,t_1]\times\mathbb{R}^n\rightarrow\mathbb{R}^m$ satisfying 
$$\bbE^{t_0,x}_{u}\left[ \int_{t_0}^{t_1} |u(t,X_t)|^q\,dt\right] <\infty $$
for all $x\in\mathbb{R}^n$ and $q\geq 1$, where $\mathbb{E}_u^{t,x}$ denotes conditional expectations w.r.t. the law of (\ref{montecarlosystemcontrol}) on the event $X_t=x\in\mathbb{R}^n$.
These conditions are sufficient for the existence of a unique, continuous, (strong) solution to (\ref{montecarlosystemcontrol}); 
see for e.g. \cite{arnold1974stochastic, touzi2012optimal}.

For each $(t,x)\in[0,T]\times\mathbb{R}^n$ and $u\in\mathcal{U}_{[t,T]}$, we associate the process defined by 
(\ref{montecarlosystemcontrol}) with the following cost functional
\begin{equation}\label{cost_functional}
	w(t,x,u) = \bbE_u^{t,x}\left[\phi(X_{T}) + \int_{t}^{T} \left\lbrace\ell(X_{s}) + u(s,X_s)^\top Ru(s,X_s)\right\rbrace\,ds\right],
\end{equation}
where $\phi,\ell:\bbR^n\to[0,\infty)$ are continuous terminal and running cost functions, respectively, and $R\in\mathbb{R}^{m\times m}$ is a positive-definite symmetric. We then define the value function as
\begin{equation}\label{value}
	v(t,x) = \inf_{u\in\calU_{[t,T]}}w(t, x, u).
\end{equation}
If there exists a unique minimizer of (\ref{value}) across the entire time horizon $t\in[0,T]$, we will refer to it as the 
optimal control and denote it by $u^*:[0,T]\times\mathbb{R}^n\rightarrow\mathbb{R}^m$.

For any suitably smooth function $\varphi:[0,T]\times\mathbb{R}^n\rightarrow\mathbb{R}$, we will denote its 
partial derivative w.r.t. the time by $\partial_t\varphi:[0,T]\times\mathbb{R}^n\rightarrow\mathbb{R}$, its gradient w.r.t. the spatial variable by $\nabla\varphi:[0,T]\times\mathbb{R}^n\rightarrow\mathbb{R}^n$ and its Hessian by 
$\nabla^2\varphi:[0,T]\times\mathbb{R}^n\rightarrow\mathbb{R}^{n\times n}$. For any $A\in\mathbb{R}^{n\times n}$, we write its 
trace as $\mathrm{tr}(A)$.
Under appropriate conditions, the value function (\ref{value}) can be associated with the following 
Hamilton-Jacobi-Bellman (HJB) equation 
\begin{align}\label{hjb0}
	-\partial_tv(t,x) = \inf_{u\in\calU_{[0,T]}} \left\lbrace \ell(x) + u(t,x)^\top R u(t,x) + u(t,x)^\top e(x)^\top\nabla v(t,x) 
	+f(x)^\top\nabla v(t,x) + \frac{1}{2}\mathrm{tr}\left[g(x)g(x)^\top\nabla^2v(t,x)\right]\right\rbrace
\end{align}
defined for $(t,x)\in[0,T]\times\mathbb{R}^n$, with a boundary condition $v(T,\cdot) = \phi(\cdot)$ at the terminal time $T$. 
This association is made precise in the following.
\begin{ass} \label{assump1}
Suppose that the value function $v:[0,T]\times\bbR^n\rightarrow[0,\infty)$ defined in (\ref{value}) is once continuously differentiable in the time variable, twice continuously differentiable the spatial variable, and is a (classical) solution to the HJB equation (\ref{hjb0}).
\end{ass}

Sufficient conditions for this assumption to hold, in addition to the modeling hypotheses introduced thus far, can be found 
for e.g. in \cite{fleming2006controlled}. These conditions typically take the form of further regularity or boundedness assumptions on the system (\ref{montecarlosystemcontrol}) and cost functions in (\ref{cost_functional}) and/or their derivatives and are rather standard\footnote{It is noteworthy that while a classical solution to the HJB equation arising in deterministic optimal control is not typical, it is well-known \cite{krylov1972control,fleming2006controlled,krylov2008controlled,touzi2012optimal} that the stochastic optimal control problem is quite generally `more regular'.}.

Under Assumption \ref{assump1}, the value-to-go satisfies the HJB equation
\begin{align}\label{eqn:hjb1}
	-\partial_tv(t,x) = \ell(x) + u^*(t,x)^\top R u^*(t,x) + u^*(t,x)^\top e(x)^\top\nabla v(t,x) 
	+f(x)^\top\nabla v(t,x) + \frac{1}{2}\mathrm{tr}\left[g(x)g(x)^\top\nabla^2v(t,x)\right]
\end{align}
for $(t,x)\in[0,T]\times\mathbb{R}^n$, with a boundary condition $v(T,\cdot) = \phi(\cdot)$. 
From (\ref{hjb0}), we find that the corresponding optimal control is given by 
\begin{align}\label{eqn:control_form}
	u^*(t,x) = -R^{-1}e(x)^\top\nabla v(t,x) 
\end{align}
for $(t,x)\in[0,T]\times\bbR^n$. Substituting the form of the optimal control back into the HJB equation (\ref{eqn:hjb1}) gives
\begin{align*}
	-\partial_tv(t,x) = \ell(x) -\frac{1}{2}\nabla v(t,x)^\top e(x)R^{-1}e(x)^\top\nabla v(t,x)  
		+f(x)^\top\nabla v(t,x) + \frac{1}{2}\mathrm{tr}\left[g(x)g(x)^\top\nabla^2v(t,x)\right]
\end{align*}
which is a nonlinear partial differential equation defined on $[0,T]\times\bbR^n$. However, the latter can be 
simplified by considering a logarithmic transformation of the value function
\begin{align}\label{eqn:log_transform}
	\psi(t,x) = \exp\left[\frac{-v(t,x)}{\gamma}\right]
\end{align}
for $(t,x)\in[0,T]\times\bbR^n$ and some $\gamma>0$ satisfying the following assumption. 

\begin{ass} \label{assump4}
	Suppose that there exists $\gamma \in \mathbb{R}$ such that $\gamma e(x)R^{-1}e(x)^\top=g(x)g(x)^\top$ for all $x\in\mathbb{R}^d$.
\end{ass}

This assumption\footnote{The interpretation of this relationship is that along directions where the noise variance is small, the control is deemed more expensive while, conversely, in those directions in which the noise has larger variance the control is cheap \cite{kappen2005path}. Indeed, this may be desirable in practice since it forces control energy to be spent mostly in those directions in which the noise level may be problematic \cite{theodorou2010generalized}.} is standard in the path integral formulation of optimal control \cite{kappen2005path}, but it also appears more generally in the stochastic optimal control literature \cite{fleming2006controlled}. Assumption \ref{assump4} allows us to write
\begin{align*}
	-\partial_t\psi(t,x) = -\frac{1}{\gamma}\ell(x)\psi(t,x) + f(x)^\top\nabla\psi(t,x) + 
	\frac{1}{2}\mathrm{tr}\left[g(x)g(x)^\top\nabla^2\psi(t,x)\right]
\end{align*}
which is a linear partial differential equation on $[0,1]\times\mathbb{R}^n$, with boundary condition $\psi(T,\cdot) = \exp[-\phi(\cdot)/\gamma]$. By the Feynman-Kac formula, the solution is given by 
$$
	\psi(t,x) = \bbE^{t,x}\left[\exp\left\lbrace-\frac{1}{\gamma}\phi(Z_{T}) - \frac{1}{\gamma}\int_{t}^{T} \ell(Z_s)\,ds \right\rbrace \right],
$$
where $\bbE^{t,x}$ denotes conditional expectations w.r.t. the law of the uncontrolled process $\{Z_t\}$ defined by 
\begin{equation}\label{eq:diff}
	dZ_{s} = f(Z_s)ds + g(Z_s)dW_s 	
\end{equation}
with initial condition $Z_t=x\in\mathbb{R}^n$. The following result uses the relationships in (\ref{eqn:control_form}) and (\ref{eqn:log_transform}) to deduce an expression of the optimal control.

\begin{prop} \label{optimalcontrolMonteCarlo}
	Suppose Assumptions \ref{assump1} and \ref{assump4}, and the modeling hypotheses hold. We have for $(t,x)\in[0,T]\times\mathbb{R}^n$
\begin{align}
\label{control formula1}
u^*(t,x) =& -R^{-1}e(x)^\top \nabla v(t,x) 
	   = \gamma R^{-1}e(x)^\top \nabla \log\psi(t,x) \nonumber \\
	   =& \lim_{r\rightarrow 0} \frac{1}{r} \frac{\bbE^{t,x}\left[\exp\left\lbrace -\frac{1}{\gamma}\left(\phi (Z_{T}) +\int_{t}^{T} \ell (Z_s)\,ds\right) \right\rbrace  \int_0^r e^{-1}(Z_s)g(Z_s)\,dW_s \right]}
	   {\bbE^{t,x}\left[\exp\left\lbrace -\frac{1}{\gamma}\left(\phi(Z_{T}) +\int_{t}^{T} \ell (Z_s)\,ds \right)\right\rbrace \right]}
\end{align}
where expectations are path integrals defined by the uncontrolled SDE (\ref{eq:diff}) with initial condition $Z_t=x$.
\end{prop}
\begin{proof}
	This result appears in \cite{thijssen2014path} with $e=g$ and it is straightforward to generalize.
\end{proof}
The controller form in Proposition \ref{optimalcontrolMonteCarlo} (and variations of such) is often referred to as the path integral formulation of optimal control \cite{kappen2005path}.

\section{Computation of optimal control}\label{sec:comp}
To simplify notation, throughout this section, we will set the terminal time as $T=1$ and consider estimating the optimal control at time $t=0$. 

\subsection{Standard approach}\label{sec:stand_appr}
From Proposition \ref{optimalcontrolMonteCarlo}, the objective is to compute for $r>0$ small 
\begin{align}\label{control_initial_time}
	u^*(0,x_0) = \frac{1}{r} \frac{\bbE^{0,x_0}\left[\exp\left\lbrace -\frac{1}{\gamma}\left(\phi (Z_{1}) +\int_{0}^{1} \ell (Z_s)\,ds\right) \right\rbrace  \int_0^r e^{-1}(Z_s)g(Z_s)\,dW_s \right]}
	   {\bbE^{0,x_0}\left[\exp\left\lbrace -\frac{1}{\gamma}\left(\phi(Z_{1}) +\int_{0}^{1} \ell (Z_s)\,ds \right)\right\rbrace \right]}.
\end{align}
We note that (\ref{control_initial_time}) neglects the additional bias incurred by truncating $r>0$ and refer 
the reader to \cite{bert} for a discussion on the impact of this parameter. 
To numerically approximate (\ref{control_initial_time}), the standard approach to path integral control would rely on a sufficiently 
precise time discretization of the model. For a sufficiently large $l\in\mathbb{N}$, we consider the Euler-Maruyama discretization 
of (\ref{eq:diff}) with step size $h=2^{-l}$
\begin{equation}\label{eqn:euler_maruyama}
Z_{kh} = Z_{(k-1)h} + f(Z_{(k-1)h})h + g(Z_{(k-1)h}) W_{k},\quad k\in\{1,\ldots,2^l\},
\end{equation}
with initial condition $Z_0=x_0$, where $W_k\sim\mathcal{N}_d(0,hI_d)$ denote independent Brownian increments that are distributed according to a Gaussian distribution mean zero and covariance $hI_d$. 
By taking $r=2^{-(M-1)}$ for $1<M\leq l$, this prompts the following time discretization of (\ref{control_initial_time}) 
\begin{align}\label{eqn:control_time_discretized}
	u^l(0,x_0) = \frac{1}{r} \frac{\bbE_l^{0,x_0}\left[\exp\left\lbrace -\frac{1}{\gamma}\left(\phi (Z_{1}) + h\sum_{k=0}^{2^l-1} \ell(Z_{kh})\right) \right\rbrace  \sum_{k=0}^{2^{l-M+1}-1} e^{-1}(Z_{kh})g(Z_{kh})W_{k+1} \right]}
	   {\bbE_l^{0,x_0}\left[\exp\left\lbrace -\frac{1}{\gamma}\left(\phi(Z_{1}) + h\sum_{k=0}^{2^l-1} \ell(Z_{kh}) \right)\right\rbrace \right]},
\end{align}
where $\bbE_l^{0,x_0}$ denotes conditional expectations w.r.t. the law of (\ref{eqn:euler_maruyama}) 
with initial condition $Z_0=x_0$.

To approximate the expectations in (\ref{eqn:control_time_discretized}), we can simulate $N$ trajectories $Z_{h:1}^i=(Z_{h}^i,\ldots,Z_1^i), i\in\{1,\ldots,N\}$ and consider the standard Monte Carlo approximation
\begin{align}\label{eqn:control_time_discretized_mc}
	u^{l,N}(0,x_0) = \frac{1}{r} \frac{N^{-1}\sum_{i=1}^N
	\exp\left\lbrace -\frac{1}{\gamma}\left(\phi (Z_{1}^i) + h\sum_{k=0}^{2^l-1} \ell(Z_{kh}^i)\right) \right\rbrace  \sum_{k=0}^{2^{l-M+1}-1} e^{-1}(Z_{kh}^i)g(Z_{kh}^i)W_{k+1}^i}
	   {N^{-1}\sum_{i=1}^N\exp\left\lbrace -\frac{1}{\gamma}\left(\phi(Z_{1}^i) + h\sum_{k=0}^{2^l-1} \ell(Z_{kh}^i) \right)\right\rbrace}.
\end{align}
Noting that we can write $u^{l,N}(0,x_0) = r^{-1}\sum_{i=1}^N\omega_l^i\varphi_l(Z_{0:r}^i)$ with normalized weights 
\begin{align*}
	\omega_l^i = \frac{\exp\left\lbrace -\frac{1}{\gamma}\left(\phi (Z_{1}^i) + h\sum_{k=1}^{2^l-1} \ell(Z_{kh}^i)\right) \right\rbrace}{\sum_{j=1}^N\exp\left\lbrace -\frac{1}{\gamma}\left(\phi (Z_{1}^j) + h\sum_{k=1}^{2^l-1} \ell(Z_{kh}^j)\right) \right\rbrace}
\end{align*}
and test function 
\begin{align*}
	\varphi_l(z_{0:r})=\sum_{k=0}^{2^{l-M+1}-1} e^{-1}(z_{kh})g(z_{kh})g^{-1}(z_{kh})
	\left\lbrace z_{(k+1)h} - z_{kh} - f(z_{kh})h\right\rbrace,
\end{align*}
it follows that (\ref{eqn:control_time_discretized_mc}) can be seen as the normalized importance sampling estimator of 
$u^l(0,x_0)=r^{-1}\mathbb{E}_{\pi^{l}}\left[\varphi_l(Z_{0:r})\right]$, where $\mathbb{E}_{\pi^{l}}$ denotes expectation 
w.r.t. the distribution
\begin{align}\label{eqn:ssm}
	\pi^l(dz_{h:1}) = G^{l}(z_{h:1}) p^{l}(dz_{h:1})
	\bigg/C^l
\end{align}
defined on $\mathbb{R}^{n2^l}$, equipped with the Borel $\sigma$-algebra $\mathcal{B}(\mathbb{R}^{n2^l})$. The notation 
$p^{l}(dz_{h:1})$ denotes the law of (\ref{eqn:euler_maruyama}) with initial condition $x_0$, 
\begin{align}\label{eqn:potentials}
	G^{l}(z_{h:1})=\prod_{k=1}^{2^l}G_k^l(z_{kh}), \quad 
	G_k^l(z_{kh}) = 
	\begin{cases}
		\exp\left\lbrace-\frac{h}{\gamma}\ell(z_{kh})\right\rbrace  &\mbox{for }k\in\{1,\ldots,2^l-1\}, \\
		\exp\left\lbrace-\frac{1}{\gamma}\phi(z_1)\right\rbrace &\mbox{for }k = 2^l,
	\end{cases}	
\end{align}
and $C^l<\infty$ is the normalizing constant of (\ref{eqn:ssm}).

Although the estimator (\ref{eqn:control_time_discretized_mc}) is straightforward to implement and amenable to 
parallel computation, its variance will often grow exponentially with the time horizon $T$ (taken as $1$ in our notation), 
and particularly so when $l\in\mathbb{N}$ is not sufficiently large \cite{jasra1}. The following section presents an alternative 
Monte Carlo approach that circumvents this difficulty. 

\subsection{Smoothing approach}\label{sec:smoothing}
Our proposed methodology is based on the observation that (\ref{eqn:ssm}) can be seen as the 
smoothing distribution of a state space model with (\ref{eqn:euler_maruyama}) as the latent Markov process 
and (\ref{eqn:potentials})
as the observation densities. This connection between optimal control and smoothing 
has been previously noted in \cite{flemingmitter1982,kappengomezopper}. This perspective explains why the variance 
of (\ref{eqn:control_time_discretized_mc}) is often large, and allows us to proposed better optimal control estimators by 
exploiting state-of the-art smoothing algorithms based on sequential Monte Carlo (SMC) methods. 

A basic SMC 
method known as the bootstrap particle filter is detailed in Algorithm \ref{alg:SMC}, where $\mathcal{R}(\omega_1,\ldots,\omega_N)$ refers to a resampling operation based on a vector of nonnegative unnormalized weights 
$\omega_i,i\in\{1,\ldots,N\}$. For example, this is the categorical distribution on $\{1,\ldots,N\}$ with probabilities 
$\omega_i/\sum_{j=1}^N\omega_j, i\in\{1,\ldots,N\}$, when multinomial resampling is employed; other
lower variance and adaptive resampling schemes can also be considered.
The algorithm requires specifying the number of particles $N_p\in\mathbb{N}$ as input, which determines the accuracy and cost of the approximation, and outputs an approximate sample from (\ref{eqn:ssm}) and an unbiased estimator of its normalizing constant. 

\begin{algorithm}
\caption{Sequential Monte Carlo for model (\ref{eqn:ssm})\label{alg:SMC}}
\textbf{Input: }number of particles $N_p\in\mathbb{N}$.
\begin{enumerate}
\item At time $0$ and particle $i\in\{1,\ldots, N_p\}$:
\begin{enumerate}
\item set $Z_{0}^{i}=x_0$;
\item set ancestor index $A_{0}^{i}=i$.
\end{enumerate}
\item For time step $k\in\{1,\ldots,2^{l}-1\}$ and particle $i\in\{1,\ldots,N_p\}$:
\begin{enumerate}
\item sample Brownian increment $W_k^i\sim\mathcal{N}_d(0,hI_d)$; 
\item set $Z_{kh}^{i} = Z_{(k-1)h}^{A_{k-1}^{i}} + f(Z_{(k-1)h}^{A_{k-1}^{i}})h + g(Z_{(k-1)h}^{A_{k-1}^{i}})W_k^i$;
\item sample ancestor index $A_{k}^{i}\sim\mathcal{R}
\Big(G_k^l(Z_{kh}^1),\ldots,G_k^l(Z_{kh}^{N_p})\Big)$.
\end{enumerate}
\item For time step $2^{l}$:
\begin{enumerate}
\item sample Brownian increment $W_{2^l}^{i}\sim\mathcal{N}_d(0,hI_d)$ for particle $i\in\{1,\ldots,N_p\}$; 
\item set $Z_{1}^{i} = Z_{(2^l-1)h}^{A_{2^l-1}^{i}} + f(Z_{(2^l-1)h}^{A_{2^l-1}^{i}})h + g(Z_{(2^l-1)h}^{A_{2^l-1}^{i}})W_{2^l}^i$;
\item sample an ancestor index $B_{2^l}\sim\mathcal{R}\Big(G_{2^l}^l(Z_{1}^1),\ldots,G_{2^l}^l(Z_{1}^{N_p})\Big)$.
\end{enumerate}
\item Trace ancestry by setting $B_{k} = A_k^{B_{k+1}}$ for $k\in\{1,\ldots,2^{l}-1\}$.
\end{enumerate}
\textbf{Output: }trajectory $Z_{h:1} = (Z_{h}^{B_1},\ldots,Z_{1}^{B_{2^l}})$ and normalizing constant estimator 
$C^{l,N_p}=\prod_{k=1}^{2^l}N_p^{-1}\sum_{i=1}^{N_p}G_k^l(Z_{kh}^i)$.
\end{algorithm}

We now consider a particular smoothing algorithm known as the particle independent Metropolis-Hastings (PIMH) method, which uses Algorithm \ref{alg:SMC} as a building block to design a MCMC method to sample from (\ref{eqn:ssm}); see \cite{andrieu} for additional details and other alternatives. An algorithmic description of PIMH is given in Algorithm \ref{alg:PIMH}. From the Markov chain $Z_{h:1}^i,i\in\{1,\ldots,N_l\}$ generated by PIMH, one obtains a consistent estimator of the optimal control for the time discretized model (\ref{eqn:control_time_discretized}) 
\begin{align}\label{eqn:single_level_estimator}
	u^{l,N_l}(0,x_0) = r^{-1}N_l^{-1}\sum_{i=1}^{N_{l}}\varphi_l(Z_{0:r}^i)
\end{align}
as the number of iterations $N_l\rightarrow\infty$, for any number of particles $N_p\in\mathbb{N}$. The latter also impacts the quality of the approximation: since the normalizing constant estimator given by Algorithm \ref{alg:SMC} is consistent as $N_p\rightarrow\infty$, the acceptance probability in Step 2(b) would be close to one if $N_p$ is large.

\begin{algorithm}
\caption{Particle independent Metropolis-Hastings for model (\ref{eqn:ssm})\label{alg:PIMH}}
\textbf{Input: }number of particles $N_p\in\mathbb{N}$ and iterations $N_{l}\in\mathbb{N}$.
\begin{enumerate}
\item Initialization:
\begin{enumerate}
\item run Algorithm \ref{alg:SMC} to obtain trajectory $Z_{h:1}^0$ and normalizing constant estimator $C^{l,N_p}_0$.
\end{enumerate}
\item For iteration $i\in\{1,\ldots,N_l\}$:
\begin{enumerate}
\item run Algorithm \ref{alg:SMC} to obtain trajectory $Z_{h:1}^*$ and normalizing constant estimator $C^{l,N_p}_*$;
\item with probability $\min\left\lbrace1,C^{l,N_p}_*/C^{l,N_p}_{i-1}\right\rbrace$ set $Z_{h:1}^i=Z_{h:1}^*$ and 
$C^{l,N_p}_i=C^{l,N_p}_*$; 
\item otherwise set $Z_{h:1}^i=Z_{h:1}^{i-1}$ and $C^{l,N_p}_i=C^{l,N_p}_{i-1}$.
\end{enumerate}

\end{enumerate}

\textbf{Output: }trajectories $Z_{h:1}^i,i\in\{1,\ldots,N_l\}$.
\end{algorithm}

\subsection{Multilevel approach}\label{sec:MLMC}
To further improve the computational efficiency of (\ref{eqn:single_level_estimator}), we now leverage upon the 
MLMC method \cite{giles,giles1,hein} by considering a hierarchy of time discretizations with time steps $h_l=2^{-l},l\in\{M,\ldots,L\}$ for $M<L$. The multilevel approach is based on the following telescopic sum 
\begin{align}
	\mathbb{E}_{\pi^L}\left[\varphi_L(Z_{0:r})\right] = \mathbb{E}_{\pi^M}\left[\varphi_M(Z_{0:r})\right] 
	+ \sum_{l=M+1}^{L}\bigg\{ \mathbb{E}_{\pi^l}\left[\varphi_l(Z_{0:r})\right] - 
	\mathbb{E}_{\pi^{l-1}}\left[\varphi_{l-1}(Z_{0:r})\right]\bigg\},
\end{align}
where $\pi^l$ refers to the distribution defined in (\ref{eqn:ssm}) with time step $h=h_l,l\in\{M,\ldots,L\}$. 
The first term in the sum $\mathbb{E}_{\pi^M}\left[\varphi_M(Z_{0:r})\right]$ can be approximated using 
the methodology described in Section \ref{sec:smoothing}, i.e. we employ Algorithm \ref{alg:PIMH} with $N_M\in\mathbb{N}$ iterations to obtain a Markov chain $Z_{h_M:1}^i,i\in\{1,\ldots,N_M\}$, 
and return the estimator 
\begin{align}\label{eqn:initial_level_estimator}
	u^{M,N_M}(0,x_0) = r^{-1}N_M^{-1}\sum_{i=1}^{N_{M}}\varphi_M(Z_{0:r}^i).
\end{align}
In standard MLMC \cite{giles}, the summands 
\begin{align}\label{eqn:summands}
	\mathbb{E}_{\pi^l}\left[\varphi_l(Z_{0:r})\right] - 
	\mathbb{E}_{\pi^{l-1}}\left[\varphi_{l-1}(Z_{0:r})\right]
\end{align}
are then estimated independently, for $l\in\{M+1,\ldots,L\}$, by sampling from appropriately constructed couplings of $\pi_l$ and $\pi_{l-1}$ that induce sufficient correlation to reduce the computational cost of just approximating $\mathbb{E}_{\pi^L}\left[\varphi_L(Z_{0:r})\right]$ in isolation. However, in our context, exact sampling from these smoothing distributions and their couplings is not feasible. In the following, we will adapt the ideas in \cite{jasra} to our setup to approximate (\ref{eqn:summands}). 

We now construct a smoothing distribution 
\begin{align}\label{eqn:ssm_extended}
\pi^{l,l-1}(dz_{h_l:1}(l),dz_{h_{l-1}:1}(l-1)) = \check{G}^{l,l-1}(z_{h_l:1}(l),z_{h_{l-1}:1}(l-1))
\,\,p^{l,l-1}(dz_{h_l:1}(l),dz_{h_{l-1}:1}(l-1))\bigg/C^{l,l-1}
\end{align}
defined on the product space $\mathbb{R}^{n2^l}\times \mathbb{R}^{n2^{l-1}}$, 
equipped with the product $\sigma$-algebra $\mathcal{B}(\mathbb{R}^{n2^l})\times\mathcal{B}(\mathbb{R}^{n2^{l-1}})$, 
that would allow us to couple our approximation of the smoothing distributions at time discretization levels $l$ and $l-1$. This corresponds to a new state space model: $p^{l,l-1}$ denotes the law of a latent process $(Z_{h_l:1}(l), Z_{h_{l-1}:1}(l-1))$ that evolves according to 
\begin{align}\label{eqn:coupled_euler_maruyama}
Z_{kh_l}(l) &= Z_{(k-1)h_l}(l) + f(Z_{(k-1)h_l}(l))h_l + g(Z_{(k-1)h_l}(l)) W_{k}(l),\quad k\in\{1,\ldots,2^l\},\\
Z_{kh_{l-1}}(l-1) &= Z_{(k-1)h_{l-1}}(l-1) + f(Z_{(k-1)h_{l-1}}(l-1))h_{l-1} + g(Z_{(k-1)h_{l-1}}(l-1)) W_{k}(l-1),\quad k\in\{1,\ldots,2^{l-1}\},\notag
\end{align}
with initial condition $Z_0(l)=Z_0(l-1)=x_0$, 
\begin{align*}
\check{G}^{l,l-1}(z_{h_l:1}(l),z_{h_{l-1}:1}(l-1)) = \prod_{k\in K_l^1}\check{G}_k^l(z_{kh_l}(l))
\prod_{k\in K_l^2}
\check{G}_k^l(z_{kh_l}(l),z_{kh_{l-1}/2}(l-1))
\end{align*}
where the observation densities $(\check{G}_1^l,\ldots,\check{G}_{2^l}^l)$ given by 
\begin{align}\label{eqn:check_obs}
\check{G}_k^l(z_{kh_l}(l)) &=  G_k^l(z_{kh_l}(l)) + 1,\quad k\in K_l^1=\{1,3,\ldots,2^l-1\}, \\
\check{G}_k^l(z_{kh_l}(l),z_{kh_{l-1}/2}(l-1)) &= \max\left\lbrace G_k^l(z_{kh_l}(l)), G_{k/2}^{l-1}(z_{kh_{l-1}/2}(l-1))
\right\rbrace,\quad k\in K_l^2=\{2,4,\ldots,2^l\},\notag
\end{align}
and $C^{l,l-1}<\infty$ is the normalization constant.
In (\ref{eqn:coupled_euler_maruyama}), the Brownian increments $(W_1(l),\ldots,W_{2^l}(l))$ and $(W_1(l-1),\ldots,W_{2^{l-1}}(l-1))$ at levels $l$ and $l-1$, respectively, are coupled by independently sampling $W_k(l)\sim\mathcal{N}_d(0,h_lI_d)$ for $k\in\{1,\ldots,2^l\}$ and setting $W_k(l-1)= W_{2(k-1)+1}(l)+W_{2k}(l)\sim\mathcal{N}_d(0,h_{l-1}I_d)$ for $k\in\{1,\ldots,2^{l-1}\}$. 
To approximate (\ref{eqn:summands}), we will rely on the identity 
\begin{align}\label{eqn:diff_levels}
\mathbb{E}_{\pi^l}\left[\varphi_l(Z_{0:r})\right] - 
\mathbb{E}_{\pi^{l-1}}\left[\varphi_{l-1}(Z_{0:r})\right] =
\end{align}
$$
\frac{\mathbb{E}_{\pi^{l,l-1}}\Big[\varphi_l(Z_{0:r}(l))
\check{H}^{l,1}(Z_{h_l:1}(l),Z_{h_{l-1}:1}(l-1))\Big]}
{\mathbb{E}_{\pi^{l,l-1}}\Big[
\check{H}^{l,1}(Z_{h_l:1}(l),Z_{h_{l-1}:1}(l-1))\Big]} - 
\frac{\mathbb{E}_{\pi^{l,l-1}}\Big[\varphi_{l-1}(Z_{0:r}(l-1))
\check{H}^{l,2}(Z_{h_l:1}(l),Z_{h_{l-1}:1}(l-1))\Big]}
{\mathbb{E}_{\pi^{l,l-1}}\Big[
\check{H}^{l,2}(Z_{h_l:1}(l),Z_{h_{l-1}:1}(l-1))\Big]} 
$$
where $\mathbb{E}_{\pi^{l,l-1}}$ denotes expectation w.r.t. $\pi^{l,l-1}$ and 
\begin{align}\label{eqn:RNs}
	\check{H}^{l,1}(z_{h_l:1}(l),z_{h_{l-1}:1}(l-1)) &= \frac{G^{l}(z_{h_l:1}(l))}{\check{G}^{l,l-1}(z_{h_l:1}(l),z_{h_{l-1}:1}(l-1))}, \\
	\check{H}^{l,2}(z_{h_l:1}(l),z_{h_{l-1}:1}(l-1)) &= \frac{G^{l-1}(z_{h_{l-1}:1}(l-1))}{\check{G}^{l,l-1}(z_{h_l:1}(l),z_{h_{l-1}:1}(l-1))}.
\end{align}
The choice of observation densities in (\ref{eqn:check_obs}) is such that the Radon-Nikodym derivatives 
(\ref{eqn:RNs}) are upper-bounded by a finite constant that does not depend on the level $l$. 
This ensures that the change of measure approach in (\ref{eqn:diff_levels}) would not introduce too much 
variance relative to exact sampling from a dependent coupling of $\pi_l$ and $\pi_{l-1}$. 

Since (\ref{eqn:ssm_extended}) is just another smoothing distribution on an extended state space, 
we can also employ the methodology of Section \ref{sec:smoothing} to construct a MCMC method to sample from it. 
The resulting algorithm based on the PIMH method is detailed in Algorithm \ref{alg:PIMH_extended}. 
Like before, this uses a SMC method on model (\ref{eqn:ssm_extended}), described in Algorithm \ref{alg:SMC_extended}, as a building block. Given the Markov chain $Z_{h_l:1}^{i}(l), Z_{h_{l-1}:1}^{i}(l-1), i\in\{1,\ldots,N_p\}$ generated by Algorithm \ref{alg:PIMH_extended}, we approximate (\ref{eqn:diff_levels}) using the estimator 
$$
r\bigg\{u^{l,N_{l}}(0,x_0) - u^{l-1,N_{l}}(0,x_0)\bigg\}
 = 
$$
$$
\frac{N_l^{-1}\sum_{i=1}^{N_l}\varphi_l(Z_{0:r}^i(l))\check{H}^{l,1}(Z_{h_l:1}^i(l),Z_{h_{l-1}:1}^i(l-1))}
{N_l^{-1}\sum_{i=1}^{N_l}\check{H}^{l,1}(Z_{h_l:1}^i(l),Z_{h_{l-1}:1}^i(l-1))} 
- 
\frac{N_l^{-1}\sum_{i=1}^{N_l}\varphi_{l-1}(Z_{0:r}^i(l-1))
\check{H}^{l,2}(Z_{h_l:1}^i(l),Z_{h_{l-1}:1}^i(l-1))}
{N_l^{-1}\sum_{i=1}^{N_l}\check{H}^{l,2}(Z_{h_l:1}^i(l),Z_{h_{l-1}:1}^i(l-1))}. 
$$
Using the above approach independently for levels $l\in\{M+1,\ldots,L\}$ and (\ref{eqn:initial_level_estimator}) 
for level $l=M$ gives the following multilevel estimator of the optimal control (\ref{control_initial_time}) 
\begin{align}\label{eqn:ML_estimator}
u^{M:L,N_{M:L}}(0,x_0) = u^{M,N_M}(0,x_0) + \sum_{l=M+1}^L \bigg\{u^{l,N_{l}}(0,x_0) - u^{l-1,N_{l}}(0,x_0)\bigg\}.
\end{align}
It follows under mild assumptions that (\ref{eqn:ML_estimator}) is a consistent estimator. In the next section, we will 
establish, under appropriate assumptions, that the multilevel estimator has a reduced computational cost 
compared to the single level estimator (\ref{eqn:single_level_estimator}). 

\begin{algorithm}
\caption{Sequential Monte Carlo for model (\ref{eqn:ssm_extended}) at level $l\in\{M+1,\ldots,L\}$\label{alg:SMC_extended}}
\textbf{Input: }number of particles $N_p\in\mathbb{N}$.
\begin{enumerate}
\item At time $0$ and particle $i\in\{1,\ldots, N_p\}$:
\begin{enumerate}
\item set $Z_{0}^{i}(l)=x_0$ and $Z_{0}^{i}(l-1)=x_0$;
\item set ancestor index $A_{0}^{i}=i$.
\end{enumerate}
\item For time step $k\in\{1,\ldots,2^{l}-1\}$ and particle $i\in\{1,\ldots,N_p\}$:
\begin{enumerate}
\item sample Brownian increment $W_k^i(l)\sim\mathcal{N}_d(0,h_lI_d)$ for level $l$; 
\item set $Z_{kh_l}^{i}(l) = Z_{(k-1)h_l}^{A_{k-1}^{i}}(l) + f(Z_{(k-1)h_l}^{A_{k-1}^{i}}(l))h_l + g(Z_{(k-1)h_l}^{A_{k-1}^{i}}(l))W_k^i(l)$;
\item if $k\in\{2,4,\ldots,2^{l}-2\}$:
\begin{enumerate}
	\item set $n=k/2$;
	\item set Brownian increment $W_n^i(l-1)= W_{2(n-1)+1}^i(l)+W_{2n}^i(l)$ for level $l-1$; 	
	\item set $Z_{nh_{l-1}}^i(l-1) = Z_{(n-1)h_{l-1}}^{A_{2(n-1)}^i}(l-1) + f(Z_{(n-1)h_{l-1}}^{A_{2(n-1)}^i}(l-1))h_{l-1} 
	+ g(Z_{(n-1)h_{l-1}}^{A_{2(n-1)}^i}(l-1)) W_{n}^i(l-1)$;
	\item sample ancestor index $A_{k}^{i}\sim\mathcal{R}
\Big(\check{G}_k^l(Z_{kh_l}^1(l),Z_{nh_{l-1}}^1(l-1)),\ldots, \check{G}_k^l(Z_{kh_l}^{N_p}(l),Z_{nh_{l-1}}^{N_p}(l-1))\Big)$.
\end{enumerate}
\end{enumerate}
\item For time step $2^{l}$:
\begin{enumerate}
\item sample Brownian increment $W_{2^l}^i(l)\sim\mathcal{N}_d(0,hI_d)$ for level $l$ and particle $i\in\{1,\ldots,N_p\}$; 
\item set $Z_{1}^{i}(l) = Z_{(2^l-1)h_l}^{A_{2^l-1}^{i}}(l) + f(Z_{(2^l-1)h_l}^{A_{2^l-1}^{i}}(l))h_l + g(Z_{(2^l-1)h_l}^{A_{2^l-1}^{i}}(l))W_{2^l}^i(l)$;\item set Brownian increment $W_{2^{l-1}}^i(l-1)= W_{2^{l}-1}^i(l)+W_{2^l}^i(l)$ for level $l-1$ and particle $i\in\{1,\ldots,N_p\}$; 
\item set $Z_{1}^i(l-1) = Z_{(2^{l-1}-1)h_{l-1}}^{A_{2^{l}-2}^i}(l-1) + f(Z_{(2^{l-1}-1)h_{l-1}}^{A_{2^{l}-2}^i}(l-1))h_{l-1} 
	+ g(Z_{(2^{l-1}-1)h_{l-1}}^{A_{2^{l}-2}^i}(l-1)) W_{2^{l-1}}^i(l-1)$;
\item sample an ancestor index $B_{2^l}\sim\mathcal{R}\Big(\check{G}_{2^l}^l(Z_{1}^1(l),Z_{1}^1(l-1)),\ldots, \check{G}_{2^l}^l(Z_{1}^{N_p}(l),Z_{1}^{N_p}(l-1))\Big)$.
\end{enumerate}
\item Trace ancestry by setting $B_{2k} = A_{2k}^{B_{2(k+1)}}$ for $k\in\{1,\ldots,2^{l-1}-1\}$ and $B_{2k-1}=B_{2k}$ for $k\in\{1,\ldots,2^{l-1}\}$.
\end{enumerate}
\textbf{Output: }trajectories $Z_{h_l:1}(l) = (Z_{h_l}^{B_1},\ldots,Z_{1}^{B_{2^l}})$ and $Z_{h_{l-1}:1}(l-1)=(Z_{h_{l-1}}^{B_2},\ldots,Z_{1}^{B_{2^{l}}})$,
and normalizing constant estimator 
$C^{l,l-1,N_p}=\prod_{k\in K_l^1}N_p^{-1}\sum_{i=1}^{N_p}\check{G}_k^l(Z_{kh_l}^i(l))
\prod_{k\in K_l^2}N_p^{-1}\sum_{i=1}^{N_p}
\check{G}_k^l(Z_{kh_l}^i(l),Z_{kh_{l-1}/2}^i(l-1))$.
\end{algorithm}

\begin{algorithm}
\caption{Particle independent Metropolis-Hastings for model (\ref{eqn:ssm_extended})\label{alg:PIMH_extended}}
\textbf{Input: }number of particles $N_p\in\mathbb{N}$ and iterations $N_{l}\in\mathbb{N}$.
\begin{enumerate}
\item Initialization:
\begin{enumerate}
\item run Algorithm \ref{alg:SMC_extended} to obtain trajectories $Z_{h_l:1}^0(l)$ and $Z_{h_{l-1}:1}^0(l-1)$, 
and normalizing constant estimator $C^{l,l-1,N_p}_0$.
\end{enumerate}
\item For iteration $i\in\{1,\ldots,N_l\}$:
\begin{enumerate}
\item run Algorithm \ref{alg:SMC_extended} to obtain trajectories $Z_{h_l:1}^*(l)$ and $Z_{h_{l-1}:1}^*(l-1)$, 
and normalizing constant estimator $C^{l,l-1,N_p}_*$;
\item with probability $\min\left\lbrace1,C^{l,l-1,N_p}_*/C^{l,l-1,N_p}_{i-1}\right\rbrace$ set $Z_{h_l:1}^{i}(l) = Z_{h_l:1}^*(l)$, 
$Z_{h_{l-1}:1}^{i}(l-1)= Z_{h_{l-1}:1}^*(l-1)$ and $C^{l,l-1,N_p}_i=C^{l,l-1,N_p}_*$; 
\item otherwise set 
$Z_{h_l:1}^{i}(l) = Z_{h_l:1}^{i-1}(l)$, 
$Z_{h_{l-1}:1}^{i}(l-1) = Z_{h_{l-1}:1}^{i-1}(l-1)$ and $C^{l,l-1,N_p}_i=C^{l,l-1,N_p}_{i-1}$.
\end{enumerate}

\end{enumerate}

\textbf{Output: }trajectories $Z_{h_l:1}^{i}(l), Z_{h_{l-1}:1}^{i}(l-1), i\in\{1,\ldots,N_l\}$.
\end{algorithm}

\section{Theoretical results}\label{sec:theory}
In our context, to sample $\pi^{l,l-1}$ for $l\in\{M+1,\dots,L\}$ (resp.~$\pi^M$) we will generate a Markov chain on a potentially enlarged space $\mathsf{W}_l \supseteq\mathbb{R}^{n(2^l+2^{l-1})}$ with the $\sigma$-algebra $\mathscr{W}_l$
(resp.~$(\mathsf{W}_M,\mathscr{W}_M)$). 
The purpose of this construction is to facilitate the application of advanced Markov chain samplers such as in \cite{andrieu}.
We denote the associated Markov kernel as $K_l:\mathcal{W}_l\times\mathscr{W}_l\rightarrow[0,1]$, $l\in\{M,\dots,L\}$.

We will consider studying (component-wise)
\begin{equation}\label{eq:mse_def}
\mathbb{E}\Big[\Big(\Big\{u^{M:L,N_{M:L}}(0,x_0) - u^{L}(0,x_0)\Big\}
+ u^{L}(0,x_0) - u^*(0,x_0)\Big)^2\Big]
\end{equation}
where $\mathbb{E}$ denotes an expectation w.r.t.~the Markov chains that have been simulated to estimate $u^{L}(0,x_0)$. 
Our objective is to verify that the MSE \eqref{eq:mse_def} can be made, for $\epsilon>0$ given, of $\mathcal{O}(\epsilon^2)$,
with a cost that is smaller than using a single MCMC algorithm to approximate $u^{L}(0,x_0)$.

The assumptions for the following result are in Appendix \ref{app:ass_state} and the proof is in Appendix \ref{app:main_thm}.
\begin{theorem}\label{theo:main_thm}
Assume (A\ref{ass:1}-\ref{ass:2}). Then there exists a $C<+\infty$ such that for any $L> M>1$, $N_{M:L}\geq 1$:
\begin{equation}\label{eq:main_thm_ineq}
\mathbb{E}\Big[\Big(\Big\{u^{M:L,N_{M:L}}(0,x_0) - u^{L}(0,x_0)\Big\}
+ u^{L}(0,x_0) - u^*(0,x_0)\Big)^2\Big]
\leq
C\Big(\sum_{l=M}^L\frac{h_l}{N_l} + \sum_{l=M+1}^{L}\sum_{q=M+1}^L\mathbb{I}_{A}(l,q)\frac{h_l^{1/2}h_q^{1/2}}{N_lN_q} + h_L\Big)
\end{equation}
where $A=\{(l,q)\in\{M+1,\dots,L\}^2:l\neq q\}$, $\mathbb{I}$ is the indicator function.
\end{theorem}

To understand the significance of this result, suppose that one iteration of each Markov chain costs $\mathcal{O}(h_l^{-1})$.
The latter is the cost of sampling Euler approximations (note Remark \ref{rem:2} in the appendix implies that the $\mathcal{O}(h_l^{-1})$ cost per iteration is possible).
Let $\epsilon>0$ be given and set $L=\mathcal{O}(|\log(\epsilon)|)$, so that the square bias is $\mathcal{O}(\epsilon^2)$ (see \eqref{eq:bias} in the appendix, but this is the $h_L$ term
on the R.H.S.~of \eqref{eq:main_thm_ineq}).
Now choosing $N_l=\mathcal{O}(\epsilon^{-2}h_{l} L)$ as in \cite{giles} yields
\begin{eqnarray*}
\sum_{l=M}^L\frac{h_l}{N_l} & = & \mathcal{O}(\epsilon^2) \\
\sum_{l=M+1}^{L}\sum_{q=M+1}^L\mathbb{I}_{A}(l,q)\frac{h_l^{1/2}h_q^{1/2}}{N_lN_q} & = & \mathcal{O}(\epsilon^2).
\end{eqnarray*}
So, we have achieved a MSE of at most $\mathcal{O}(\epsilon^2)$. 
The cost to achieve this MSE is 
$$
\sum_{l=M}^Lh_l^{-l}N_l =  \mathcal{O}(\epsilon^{-2}\log(\epsilon)^2).
$$
If one simply used a Markov chain simulation for $\pi^L$ or using the original approach in \cite{bert} the cost to achieve this same MSE would be $\mathcal{O}(\epsilon^{-3})$; a significant increase.

\section{Numerical results}\label{sec:numerics}

\subsection{Linear quadratic Gaussian control}
We consider a linear quadratic Gaussian (LQG) control problem where the underlying continuous-time linear controlled process is given by 
\begin{equation}
dX_t = AX_t dt + Bu(t,X_t) dt + dW_t
\end{equation}
with $t \in [0,T]$, $X_t \in \mathbb{R}$, $u(t,X_t) \in \mathbb{R}$ and $W_t \in \mathbb{R}$ denoting a standard Brownian motion in $\mathbb{R}$. The estimation of the optimal control that minimizes the following quadratic cost function is considered
\begin{equation}
w(t,x,u) = \mathbb{E}_u^{t,x}\left[FX_T^2 + \int_{t}^{T}\{QX_s^2 + Ru(s,X_s)^2\}\,ds\right].
\end{equation}

In our experiments, we set $A = -1$, $B = 1$, $F = 1$, $Q = 1$, $R = 0.1$, $M = 4$ and initialize the process at $X_0 = -0.1$. The estimation of the optimal control at time $t=0$ using the different methods considered is compared:
the standard Monte Carlo (MC) approach described in Section \ref{sec:stand_appr} (see also \cite{bertoli2015nonlinear, kappen2005linear}), 
single level PIMH introduced Section \ref{sec:smoothing}, and multilevel Monte Carlo (MLMC) approach based on PIMH as detailed in Section \ref{sec:MLMC}.
For the MLMC method, we set the number of MCMC samples according to the multilevel analysis, i.e.~$N_l = \mathcal{O}(\epsilon^{-2}h_lL)$ and $L = \mathcal{O}(|\text{log}(\epsilon)|)$, with the number of particles in SMC to be fixed at $N_p = 500$. 

For the single level PIMH method, the number of MCMC samples is $N = \mathcal{O}(\epsilon^{-2})$. 
The single level PIMH algorithm and the multilevel PIMH algorithm are compared for a time horizon of length $T = 1$, with the MSE v.s. cost plot presented in Figure \ref{fg:1}. 
For the same level of MSE, the cost reduction in the multilevel approach is clear. 
All three approaches are then compared for a longer time horizon of length $T = 10$. The number of samples is set so that the computational cost of the standard MC approach and that of the PIMH method are the same. 
From Figure \ref{fg:2}, we see a marked reduction in computational cost when employing single level PIMH 
as compared to the standard MC approach, and a further reduction in cost with the multilevel PIMH approach. 

\begin{figure}
	\centering
	\captionsetup{justification=centering}
	\includegraphics[width=9cm, height=5cm]{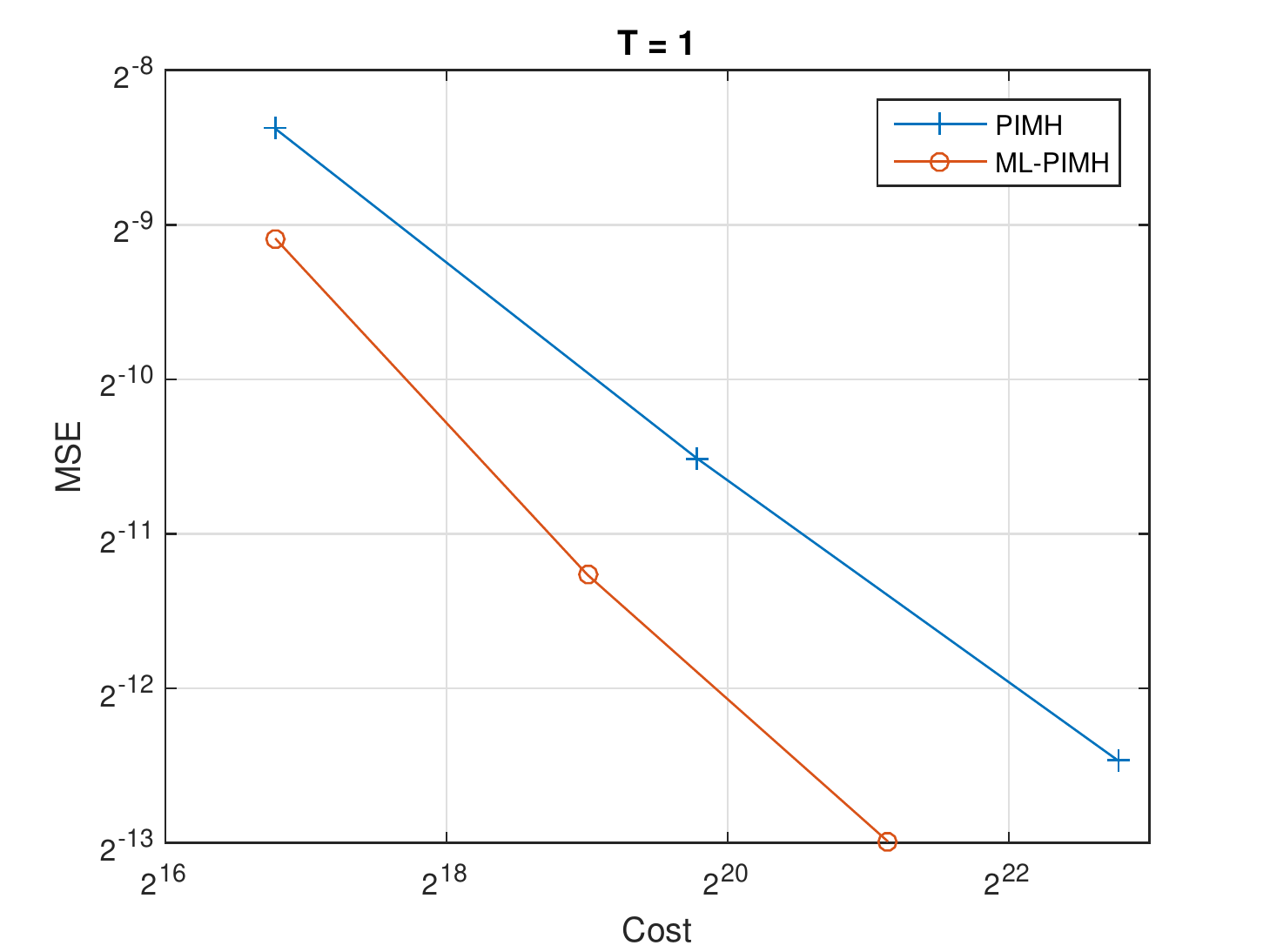}
	\caption{MSE v.s. cost for LQG model with time horizon $T=1$.}
	\label{fg:1}
\end{figure}

\begin{figure}
	\centering
	\captionsetup{justification=centering}
	\includegraphics[width=9cm, height=5cm]{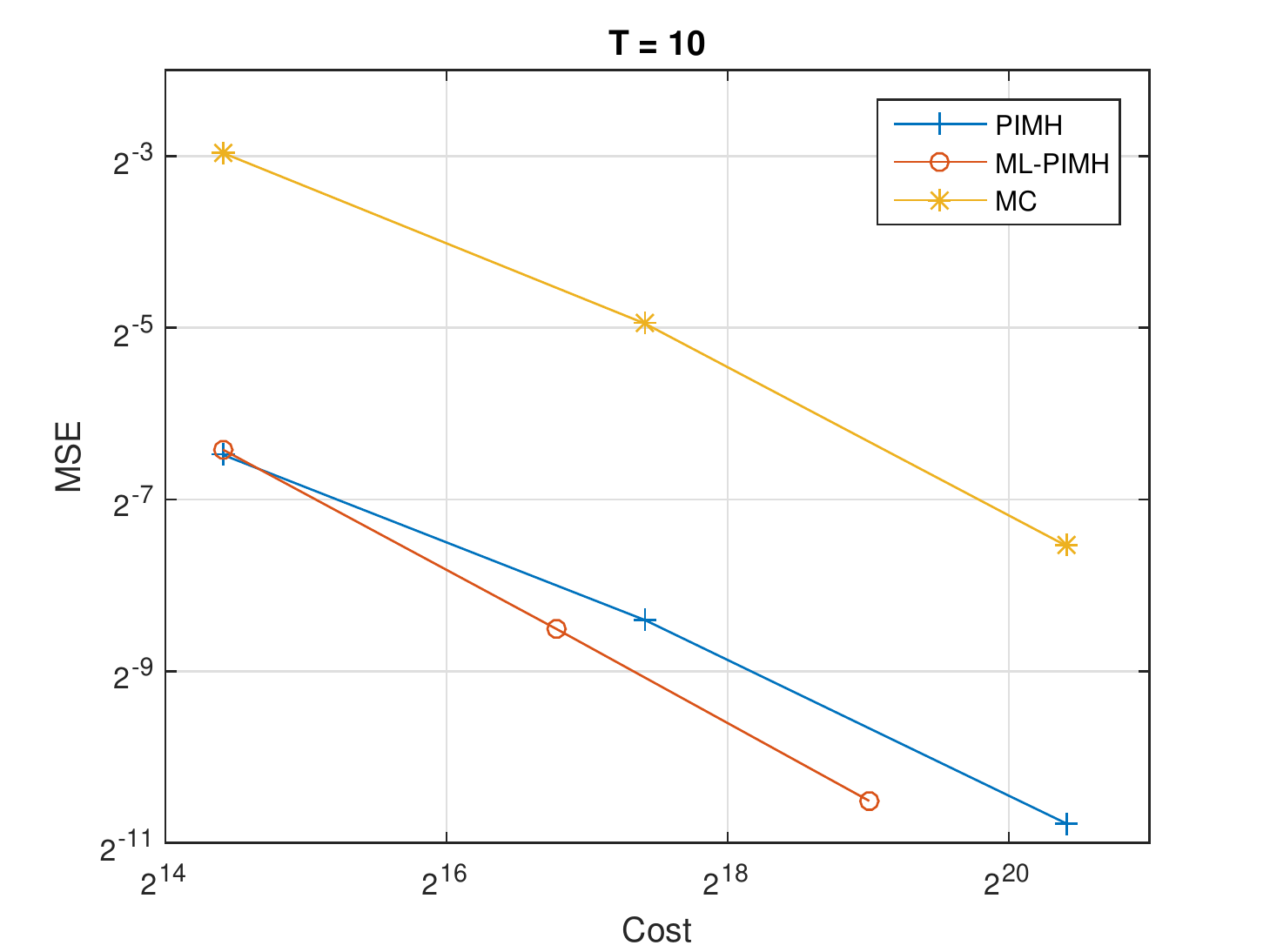}
	\caption{MSE v.s. cost for LQG model with time horizon $T=10$.}
	\label{fg:2}
\end{figure}

\subsection{Nonlinear compartmental model}

\subsubsection{Model specification}

We consider the optimal control of a stochastic compartmental model for an epidemic with cost-controlled vaccination. The state variables are $X_t=(S_t,I_t,V_t,R_t)\in\mathbb{R}^4$ corresponding to susceptible (S), infected (I), vaccinated (V), and removed (R) individuals in a population with the constraint
\begin{equation}
	S_t+I_t+V_t+R_t = 1\label{SIVRconstraint1}
\end{equation}
for all $t\geq0$. The controlled model considered herein respects this constraint for all $t>0$, whenever it is enforced to hold at time $0$; see \cite{Tornatore2014}.
We consider a modification of \cite{Tornatore2014} suitable for our purposes,
\begin{equation}
\begin{split}
	dS_t & ~=~ \left(\beta - \beta\,S_t - \kappa\,I_t\,S_t +\theta\,V_t -S_t\,u(t,X_t) \right)dt - \sigma\,S_t\,dW_t \\
	dI_t & ~=~ \left(\kappa\,S_t\,I_t+\varepsilon\,\kappa\,V_t\,I_t-\lambda\,I_t-\beta\,I_t + \varrho\,S_t\,u(t,X_t)\right)dt + \sigma\,(S_t-\varepsilon\,S_t - \sigma_\varrho\,\varrho\,S_t)\,dW_t \\
	dV_t & ~=~ \left(-\varepsilon\,\kappa\,I_t\,V_t -\beta\,V_t - \theta\,V_t + (1-\varrho)\,S_t\,u(t,X_t)\right)dt + \sigma\,(\varepsilon\,S_t + \sigma_\varrho\,\varrho\,S_t)\,dW_t \\
	dR_t & ~=~ \left(\lambda\,I_t-\beta\,R_t \right)dt 
\end{split} \label{model1}
\end{equation}
where $W_t\in\bbR$ denotes a standard Brownian motion in $\mathbb{R}$. In this model, the birth and death rate are given by $\beta\in(0,1)$, and the infection rate is controlled by $\kappa>0$, known as the contact rate. The parameter $\lambda\geq0$ controls the curing rate, $\theta\geq0$ controls the loss of vaccine effectiveness, and $\varepsilon\in(0,1)$ controls the efficacy of the vaccination protocol, i.e. letting $\varepsilon = 0$ would imply the vaccine is perfectly effective, while $\varepsilon=1$ implies the vaccination has no effect. The parameter $0<\varrho\approx0$ is necessary for our model to be well-defined and taken small enough so that it has no effect qualitatively. 

The control input $u\in\bbR$ specifies the fraction of the susceptible class being vaccinated at any moment. Note that although we would like $u\in[0,1]$, values outside this constraint cause no mathematical difficulty and do not pose a problem for satisfaction of the constraint (\ref{SIVRconstraint1}), i.e. $d(S_t+I_t+V_t+R_t)=0$.
Given some fixed terminal time $T>0$, the cost function we aim to minimize is given by
\begin{equation}\label{eqn:SIVR_value}
	w(t,x,u) = \bbE_u^{t,x}\left[I_T^2 + \int_{t}^{T} \{q\,I_s + r\,u(s,X_s)^2\}~ds\right],
\end{equation}
for $t\in[0,T]$, where $q,r>0$ are weighting parameters. Note that the running cost is linear in the state.

The basic reproduction rate is $\mathfrak{R}_0 = \kappa/(\beta+\lambda)$ and, with $\theta=\varrho=0$, it is shown in \cite{Witbooi2015} that for $\mathfrak{R}_0<1$ the stochastic system (\ref{model1}) is almost surely exponentially stable with any constant $u$ to the equilibrium
$$
	(S^*,I^*,V^*,R^*) ~=~ \left(\frac{\beta}{\beta+u},0,1-\frac{\beta}{\beta+u},0 \right).
$$
Hence, we take parameters such that $\mathfrak{R}_0>1$ going forward. In particular, we take $\beta = 0.016$, $\kappa = 0.55$, $\lambda = 0.45$ and $\varepsilon =0.4$. We also take $\theta=0.1$, $\varrho=0.01$ and $\sigma=0.4$. The initial condition is $(S_0,I_0,V_0,R_0) = (0.75,0.15,0.05,0.05)$. 
To check our assumptions stated in Section \ref{sec:problem}, we note that
\begin{align*}
&\gamma\,r^{-1}\left(\begin{array}{c} -S_t \\  \varrho\,S_t  \\ (1-\varrho)\,S_t \\ 0\end{array}\right)\left(\begin{array}{cccc} -S_t, &  \varrho\,S_t,  & (1-\varrho)\,S_t, & 0\end{array}\right)\\
& = \sigma^2\left(\begin{array}{c} -S_t \\ (1-\varepsilon-\sigma_\varrho\,\varrho)\,S_t  \\  (\varepsilon+\sigma_\varrho\,\varrho)\,S_t \\ 0\end{array}\right)\left(\begin{array}{cccc} -S_t, & (1-\varepsilon-\sigma_\varrho\,\varrho)\,S_t,  &  (\varepsilon+\sigma_\varrho\,\varrho)\,S_t, & 0\end{array}\right)
\end{align*}
for some $\gamma\in\bbR$. Indeed, we require $\varepsilon + (\sigma_\varrho+1)\varrho = 1$ in order to find $\gamma$ uniquely. Given our prior parameters this implies $\sigma_\varrho=59$ and then $\gamma = 0.16r$.
We only need the left inverse of $(-S_t, \varrho S_t, (1-\varrho)S_t )^{\top}$. 
Note in the control computation we can act as if the system is three-dimensional and ignore the dynamics of $R_t$. The left inverse in this case is
\begin{align*}
\left(\begin{array}{c} -S_t \\  \varrho\,S_t  \\ (1-\varrho)\,S_t \end{array}\right)^{-1} = \left[\left(\begin{array}{cccc} -S_t, &  \varrho\,S_t,  & (1-\varrho)\,S_t \end{array}\right) \left(\begin{array}{c} -S_t \\  \varrho\,S_t  \\ (1-\varrho)\,S_t \end{array}\right)\right]^{-1}\left(\begin{array}{cccc} -S_t, &  \varrho\,S_t,  & (1-\varrho)\,S_t \end{array}\right)
\end{align*}
which in general exists since the inverse on the right hand side in general exists.

\subsubsection{Numerical results}

In our experiments, we set $q = 1$, $r = 0.05$, $T = 3$ and $M = 3$. 
In the following, we will compare the single level PIMH algorithm to the multilevel PIMH algorithm for the task of optimal 
control estimation.
As the standard MC approach did not perform well in the simple LQG model, it is not considered for this application. 
A trajectory of the controlled process generated by the two algorithms are shown in Figure \ref{fg:3}. 
It is clear from Figure \ref{fg:3} that the infected and the susceptible compartments are decreasing over time, while the vaccinated and the removed compartments are increasing over time.

\begin{figure}
	\centering
	\captionsetup{justification=centering}
	\includegraphics[width=8cm, height=5cm]{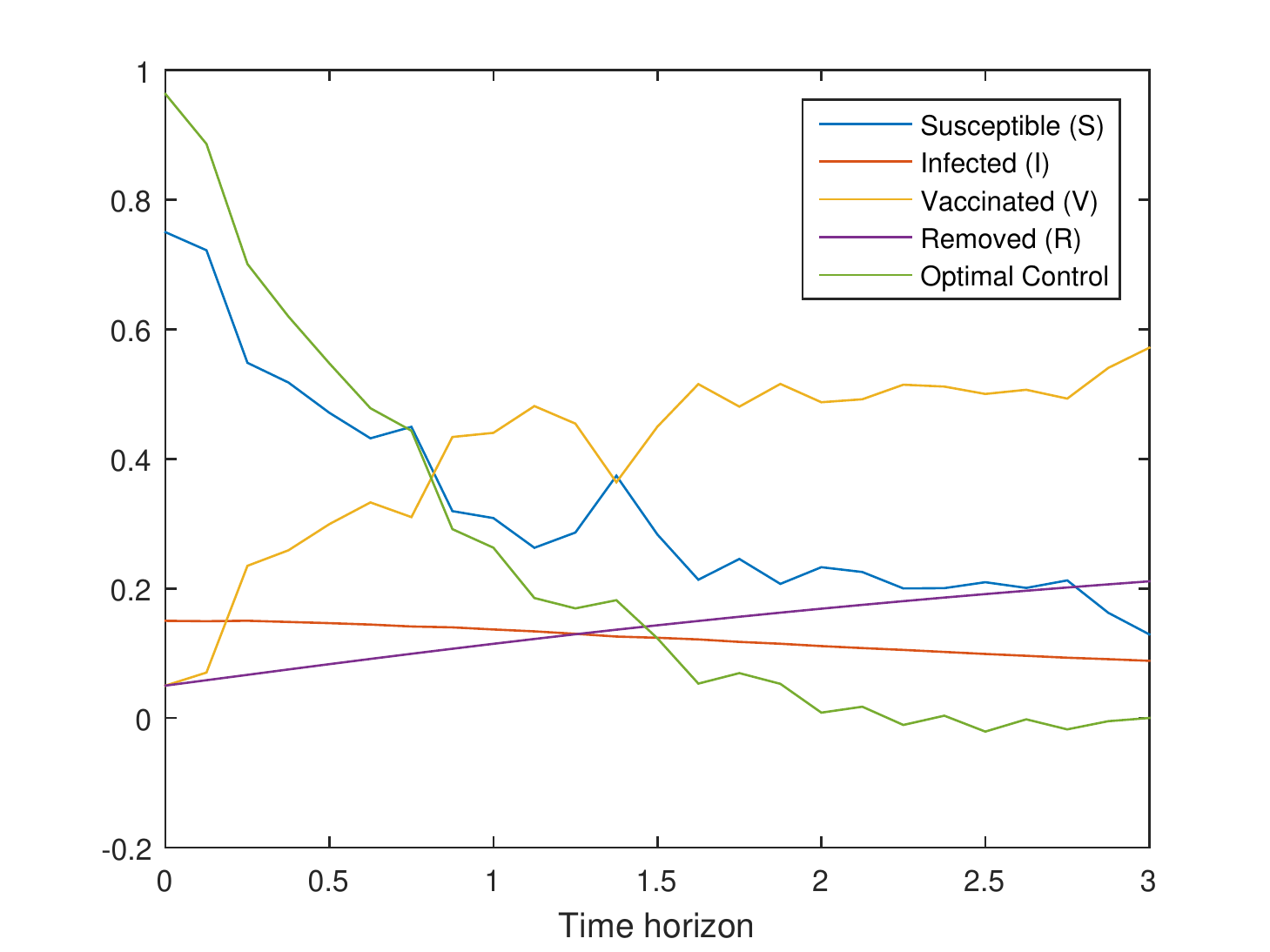} 
	\includegraphics[width=8cm, height=5cm]{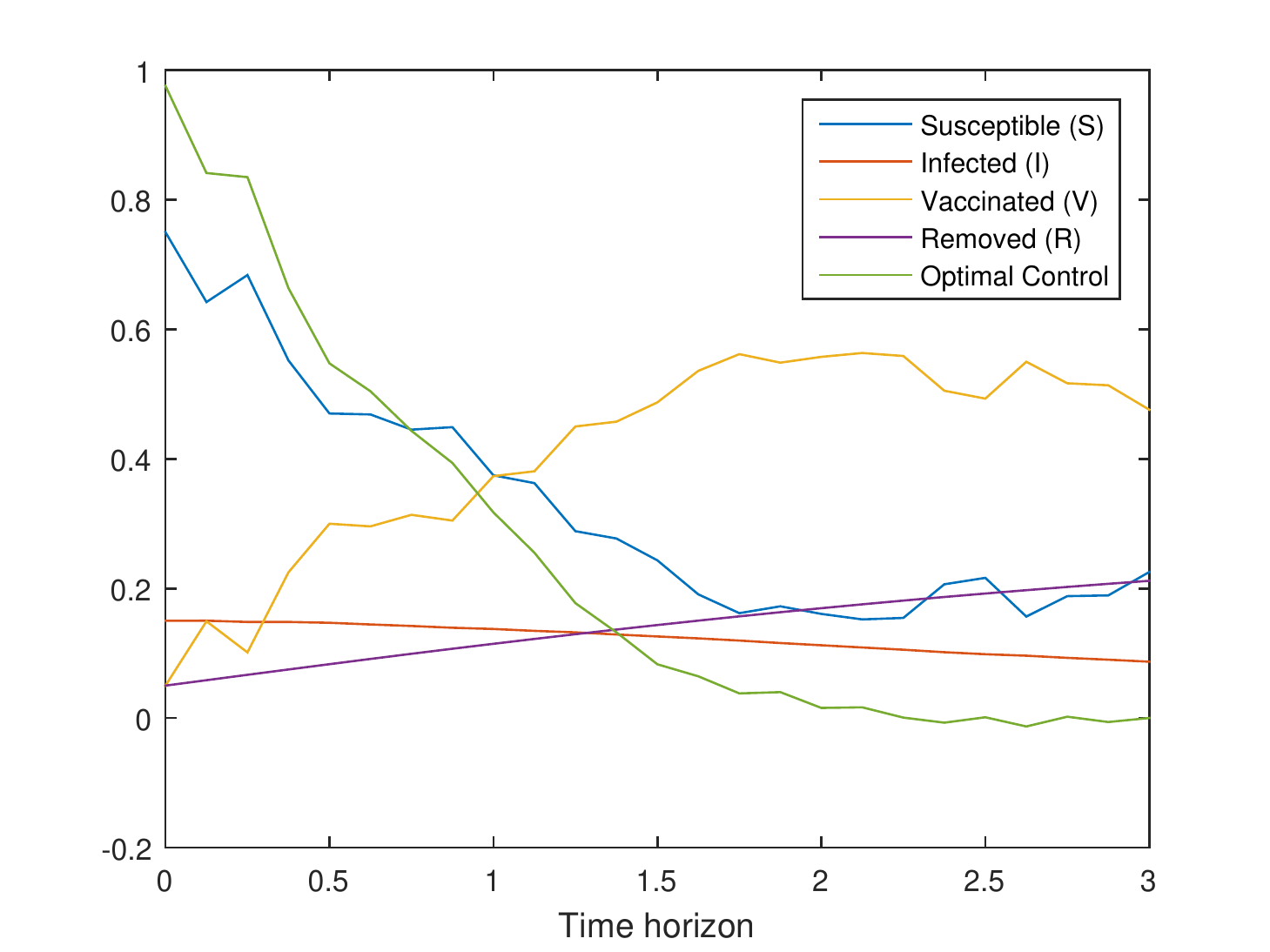}
	\caption{A trajectory of the controlled process and the corresponding approximation of the optimal control 
	generated by the single level PIMH algorithm (left) and the multilevel PIMH algorithm (right).}
	\label{fg:3}
\end{figure}

In Figure \ref{fg:4}, we compare the two algorithms at a fixed computational cost by reporting the sample average (left panel) 
and sample variance (right panel) of the value function, computed using $20$ independent repetitions of each algorithm. 
We observe that the multilevel approach achieves lower values on average in terms of the objective (\ref{eqn:SIVR_value}) with much smaller variance.

\begin{figure}
	\centering
	\captionsetup{justification=centering}
	\includegraphics[width=8cm, height=5cm]{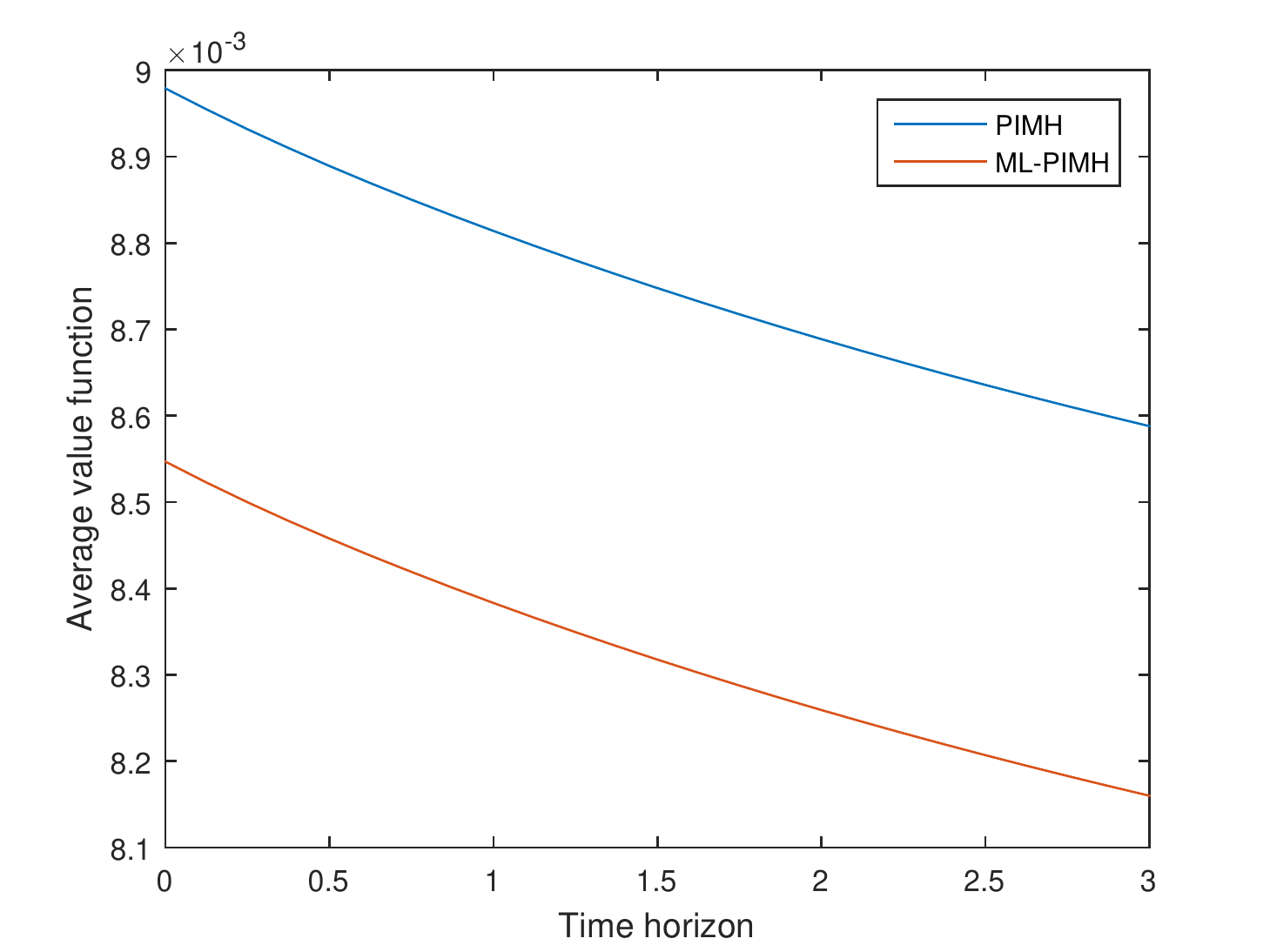} 
	\includegraphics[width=8cm, height=5cm]{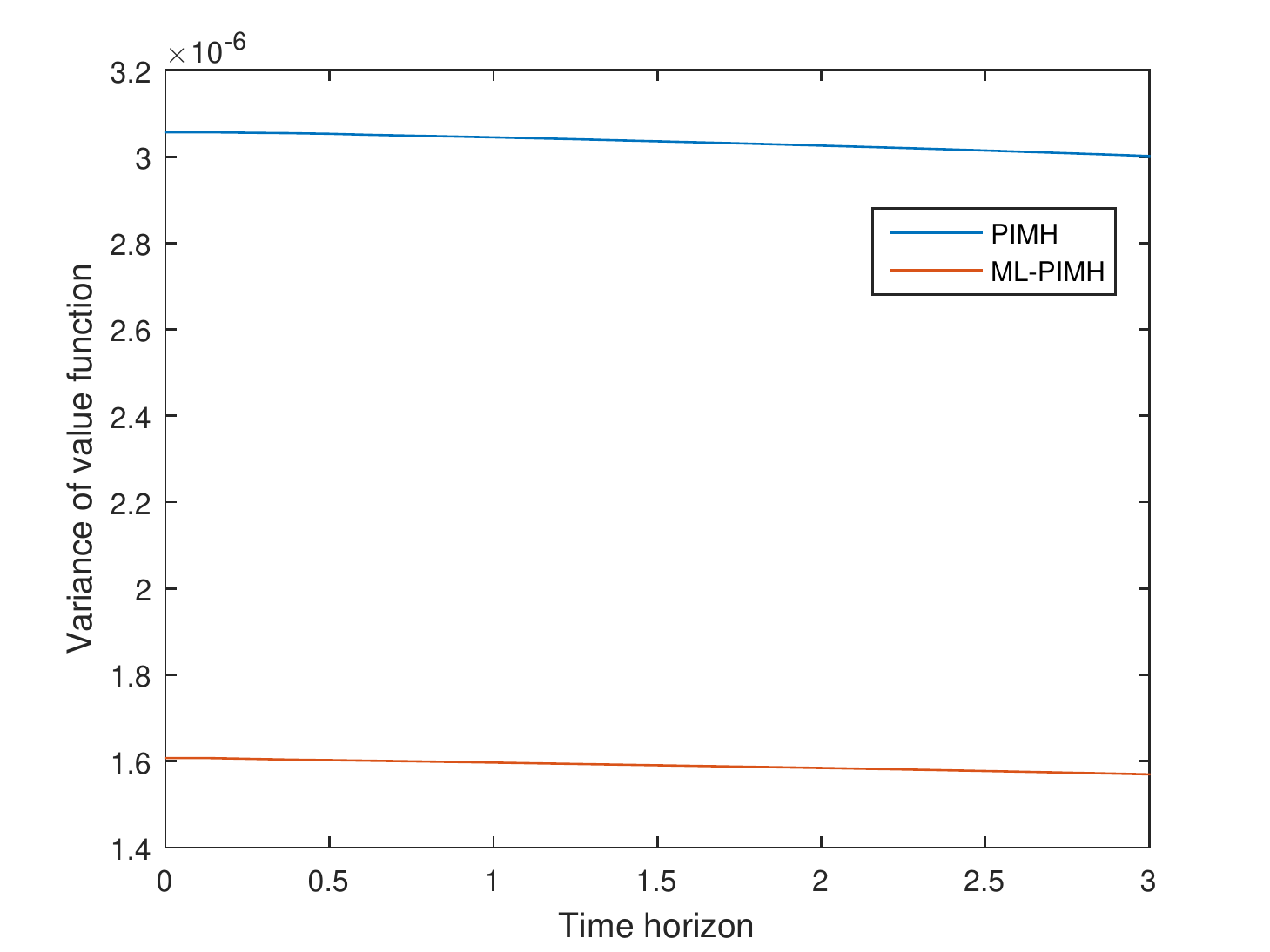}
	\caption{Sample average (left) and sample variance (right) of value function over time interval $[0,T]$.}
	\label{fg:4}
\end{figure}

Lastly, for the task of estimating the optimal control at time $t=0$, we present a MSE v.s.~cost plot in Figure \ref{fg:5}. 
For the MLMC method, we set the number of MCMC samples according to the multilevel analysis
i.e.~$N_l = \mathcal{O}(\epsilon^{-2}h_lL)$ and $L = \mathcal{O}(|\text{log}(\epsilon)|)$, with the number of particles in SMC to be fixed at $N_p = 200$. 
For the singe level PIMH method, the number of MCMC samples is taken as $N = \mathcal{O}(\epsilon^{-2})$.  
The true value is computed by running the latter algorithm at one plus the most precise level (i.e.~$L + 1$) and the MSE is computed using $20$ independent repetitions of each algorithm. The results illustrate that the multilevel approach offers 
significant reduction in computational cost.

\begin{figure}
	\centering
	\captionsetup{justification=centering}
	\includegraphics[width=9cm, height=5cm]{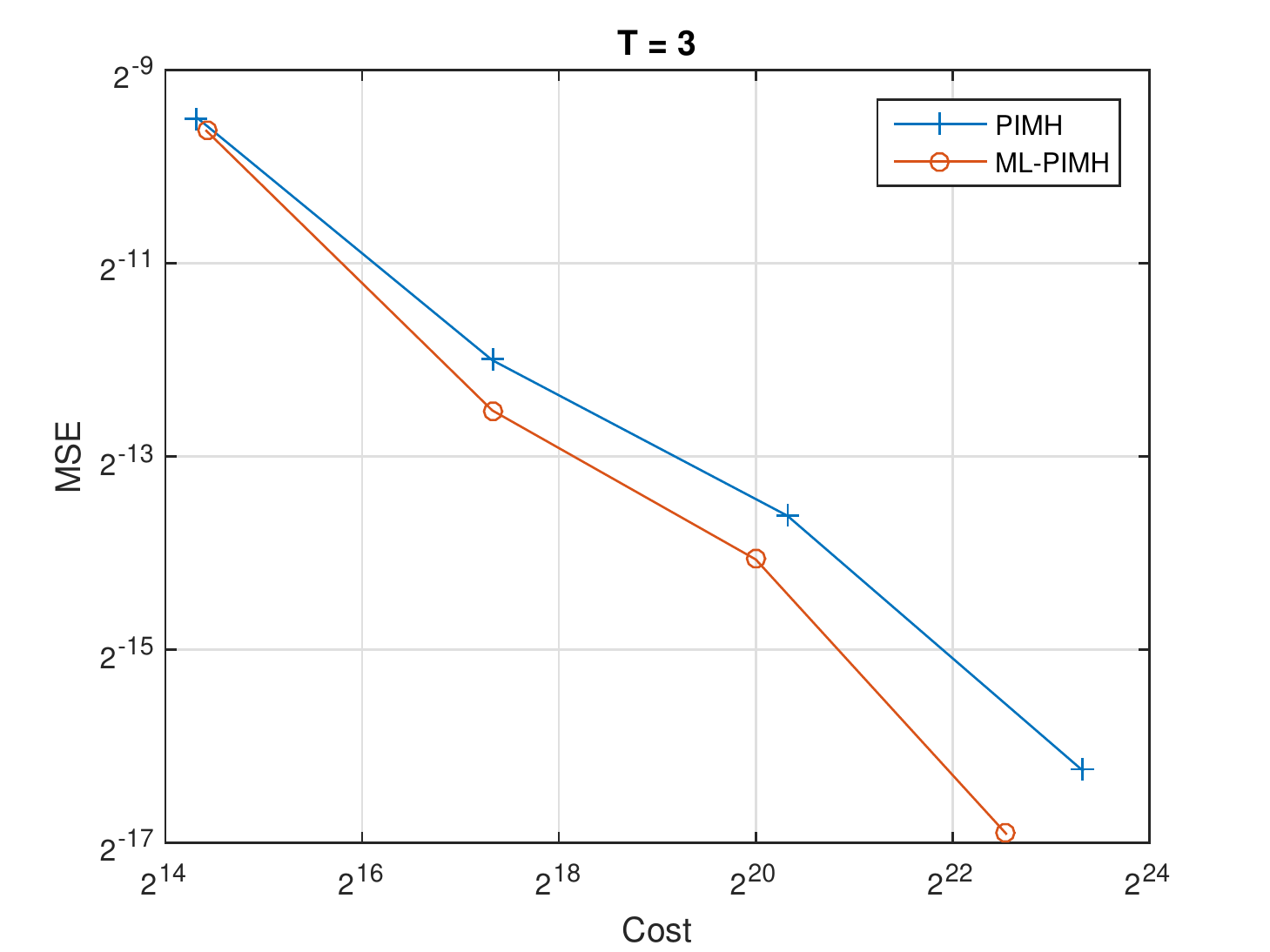}
	\caption{MSE v.s. cost for nonlinear compartmental model.}
	\label{fg:5}
\end{figure}

\subsubsection*{Acknowledgements}
AJ \& YX were supported by an AcRF tier 2 grant: R-155-000-161-112. AJ is affiliated with the Risk Management Institute, the Center for Quantitative Finance 
and the OR \& Analytics cluster at NUS. AJ was supported by a KAUST CRG4 grant ref: 2584.

\appendix

\section{Technical results}

Throughout our proofs $C$ is a finite constant that does not depend upon $l$ and whose value may change on each appearance.

\subsection{Assumptions}\label{app:ass_state}

In the context of Theorem \ref{theo:main_thm}, to shorten our proofs, we will suppose that the
Markov chain(s) are started in stationarity. This latter assumption can be removed with some work, but is unnecessary in order to convey the point of our approach.
Note also that we are assuming that $r$ is fixed throughout and there is an additional bias which is not addressed.
In order to derive our theoretical results, in addition to the assumptions that have already been made, we will make the following assumptions.
Below $\mathcal{P}(\mathsf{W}_l)$ denotes the collection of probability measures on $\mathsf{W}_l$.

\begin{hypA}\label{ass:1}
$\phi$, $\ell$, $e$, $f$, $g$ are all bounded and measurable. In addition $\phi$, $\ell$ are Lipschitz.
Set $\alpha = e^{-1}g$, then we assume each element of $\alpha$ is bounded and Lipschitz.
\end{hypA}
\begin{hypA}\label{ass:2}
$K_l$ is reversible. There exists a $\kappa\in(0,1)$ such that for each $l\in\{M,\dots,L\}$ there exists a $\nu\in\mathcal{P}(\mathsf{W}_l)$  
such that for any $\eta:\mathsf{W}_l\rightarrow\mathbb{R}$ bounded, measurable and Lipschitz, $x\in\mathsf{W}_l$
$$
\int_{\mathsf{W}_l} \eta(x') K_l(x,dx')  \geq \kappa \int_{\mathsf{W}_l}\eta(x')\nu(dx').
$$
\end{hypA}

\begin{rem}\label{rem:2}
In (A\ref{ass:2}) we have assumed a mixing rate that will be independent of $l$. At first glance, it may seem that this is not possible in practice.
However, using Lemma \ref{lem:a1} below, one can easily establish that (for example) the PIMH algorithm in \cite{andrieu} has such a property.
\end{rem}

\subsection{Proof of Theorem \ref{theo:main_thm}}\label{app:main_thm}

The proof is constructed by using several results, which are first quoted and proved later on in the appendix.

\begin{theorem}\label{theo:1}
Assume (A\ref{ass:1}-\ref{ass:2}). Then there exists a $C<\infty$ such that for any $l\in\{M+1,\dots,L\}$, $N_l\geq 1$:
$$
\mathbb{E}[(u^{l,N_{l}}(0,x_0)-u^{l-1,N_{l}}(0,x_0)-\{u^{l}(0,x_0)-u^{l-1}(0,x_0)\})^2] \leq \frac{C h_l}{N_l}.
$$
\end{theorem}

\begin{proof}
This essentially the same as \cite[Theorem 3.1]{jasra} except that one needs Proposition \ref{prop:a1} and Lemma \ref{lem:a1} in Appendix \ref{app:main_theo}.
\end{proof}

The proof of the following result is in Appendix \ref{app:prop_proof}.

\begin{prop}\label{prop:main_prop}
Assume (A\ref{ass:1}-\ref{ass:2}). Then there exists a $C<\infty$ such that for any $l\in\{M+1,\dots,L\}$, $N_l\geq 1$:
$$
|\mathbb{E}[u^{l,N_{l}}(0,x_0)-u^{l-1,N_{l}}(0,x_0)-\{u^{l}(0,x_0)-u^{l-1}(0,x_0)\}]| \leq \frac{C h_l^{1/2}}{N_l}.
$$
\end{prop}

\begin{rem}
We note that, via \cite[Proposition A.1.]{jasra}, a simple decomposition:
\begin{equation}\label{eq:simple_decomp}
\frac{a}{b}-\frac{c}{d} = \frac{a-c}{b}  + \frac{c[d-b]}{bd}
\end{equation}
for any $(a,b,c,d)\in\mathbb{R}$, $b\neq 0$, $d\neq 0$
and Lemma \ref{lem:a1} one can show that
\begin{equation}\label{eq:l_2_first_chain}
\mathbb{E}[(u^{M,N_{M}}(0,x_0)-u^{M}(0,x_0))^2] \leq \frac{C}{N_M}
\end{equation}
for $C<\infty$ independent of $N_M,l$. In addtion,  in the proof of Proposition \ref{prop:a1} we have established that 
\begin{equation}\label{eq:bias}
|u^{L}(0,x_0) - u^*(0,x_0)| \leq Ch_L^{1/2}
\end{equation}
where $C<\infty$ does not depend on $l$. \eqref{eq:bias} can be obtained using the bound for \eqref{eq:3} ($l=L$), \eqref{eq:simple_decomp} and Lemma \ref{lem:a1}.
\end{rem}

\begin{proof}[Proof of Theorem \ref{theo:main_thm}]
Using the $C_2-$inequality,
$$
\mathbb{E}\Big[\Big(\Big\{u^{M:L,N_{M:L}}(0,x_0) - u^{L}(0,x_0)\Big\}
+ u^{L}(0,x_0) - u^*(0,x_0)\Big)^2\Big]
\leq 
$$
\begin{equation}\label{eq:first_bound_mse}
2\Big(
\mathbb{E}\Big[\Big(u^{M:L,N_{M:L}}(0,x_0) - u^{L}(0,x_0)\Big)^2\Big] + |u^{L}(0,x_0) - u^*(0,x_0)|^2
\Big).
\end{equation}
First, apply the $C_2-$inequality to the variance term (the bias is the L.H.S.~of \eqref{eq:bias} and hence the variance is the left term on the R.H.S.~of \eqref{eq:first_bound_mse}),
splitting $u^{M,N_{M}}(0,x_0)-u^{M}(0,x_0)$ and the other terms. Second, apply Theorem \ref{theo:1}, Proposition \ref{prop:main_prop} and \eqref{eq:l_2_first_chain} to the variance terms and \eqref{eq:bias} to the bias term. This allows one to complete the proof.
\end{proof}

\subsection{Proofs for Theorem \ref{theo:1}}\label{app:main_theo}

Below for $i_1\in\{1,\dots,m\}$ we write the $i_1^{th}-$element of $m-$vector $\varphi_l(z_{h_l:M_lh_l})$ as
$\varphi_l(z_{h_l:M_lh_l})_{i_1}$. For $(i_1,i_2)\in\{1,\dots,m\}\times\{1,\dots,d\}$, we write the 
$i_1^{th},i_2^{th}$ element of $\alpha(z)$ as $\alpha(z)_{i_1i_2}$.
For ease of notation set
$$
T(i_1,l,q) := 
$$
$$
\mathbb{E}_{\pi^{l,l-1}}\Big[\Big(\varphi_l(Z_{h_l:M_lh_l}(l))_{i_1}\check{H}^{l,1}(Z_{h_l:1}(l),Z_{h_{l-1}:1}(l-1))-
\varphi_{l-1}(Z_{h_{l-1}:M_{l-1}h_{l-1}}(l-1))_{i_1}\check{H}^{l,2}(Z_{h_l:1}(l),Z_{h_{l-1}:1}(l-1))
\Big)^q\Big]^{3-q}.
$$

\begin{prop}\label{prop:a1}
Assume (A\ref{ass:1}). Then for any $i_1\in\{1,\dots,m\}$, $q\in\{1,2\}$, there exist a $C<+\infty$ such that for any $l\in\{M+1,\dots,L\}$
$$
T(i_1,l,q) \leq C h_l.
$$
\end{prop}

\begin{proof}
We give the proof for $q=2$, the proof for the case $q=1$ follows by Jensen's inequality.

Let $\mathcal{Z}^{l,l-1}$ denote the normalizing constant of $\pi^{l,l-1}$. By Lemma \ref{lem:a1} it easily follows that 
$$
\mathcal{Z}^{l,l-1} \geq C.
$$
Thus it follows that
\begin{equation}\label{eq:1}
T(i_1,l,2) \leq C \int_{\mathbb{R}^{d(2^l+2^{l-1})}}\Big(\varphi_l(z_{h_l:M_lh_l}(l))_{i_1} \prod_{k=0}^{h_l^{-1}} G_k^l(z_{kh_l}(l)) -
\end{equation}
$$
\varphi_{l-1}(z_{h_{l-1}:M_{l-1}h_{l-1}}(l-1))_{i_1} \prod_{k=0}^{h_{l-1}^{-1}} G_k^{l-1}(z_{kh_{l-1}}(l-1))\Big)^2 p^{l,l-1}(d(z_{h_l:1}(l),z_{h_{l-1}:1}(l-1))).
$$
Define
$$
T_1(i_1,l) :=  \mathbb{E}_D\Big[\Big(\varphi_l(z_{h_l:M_lh_l}(l))_{i_1} \prod_{k=0}^{h_l^{-1}} G_k^l(z_{kh_l}(l))-
(\sum_{i_2=1}^d\int_{0}^r\alpha(Z_s)_{i_1i_2}dW_{s}(i_2))\exp\{-\frac{1}{\gamma}(\phi(Z_1)+\int_{0}^1 \ell(Z_s)ds)\}
\Big)^2\Big]
$$
where we are denoting expectations w.r.t.~the law of the diffusion \eqref{eq:diff} as $\mathbb{E}_{D}$ and $W_{s}(i_2)$ is the $i_2^{th}-$element of
the Brownian motion in \eqref{eq:diff}. Then it is clear that the integral on the R.H.S.~of \eqref{eq:1} is upper-bounded by
$$
2(T_1(i_1,l)+T_1(i_1,l-1)).
$$
Hence we focus upon $T_1(i_1,l)$ to conclude our result. 

We have
\begin{equation}\label{eq:3}
T_1(i_l,l) \leq 2\mathbb{E}_D\Big[\Big(\varphi_l(z_{h_l:M_lh_l}(l))_{i_1} \prod_{k=0}^{h_l^{-1}} G_k^l(z_{kh_l}(l))-
\sum_{i_2=1}^d\int_{0}^r\alpha(Z_s)_{i_1i_2}dW_{s}(i_2)\exp\{-\frac{1}{\gamma}(\phi(Z_1)+h_l\sum_{k=1}^{h_l^{-1}-1} \ell(Z_{kh_l}))\}\Big)^2\Big] +
\end{equation}
$$
2\mathbb{E}_D\Big[\Big(\Big(\sum_{i_2=1}^d\int_{0}^r\alpha(Z_s)_{i_1i_2}dW_{s}(i_2)\Big)\exp\{-\frac{1}{\gamma}(\phi(Z_1)+h_l\sum_{k=1}^{h_l^{-1}-1} \ell(Z_{kh_l}))\}
-\exp\{-\frac{1}{\gamma}(\phi(Z_1)+\int_{0}^1 \ell(Z_s)ds)\}
\Big)^2\Big].
$$
We deal with the two expectations on the R.H.S.~of \eqref{eq:3} individually. For the first term on the R.H.S.~we have that
$$
\mathbb{E}_D\Big[\Big(\varphi_l(z_{h_l:M_lh_l}(l))_{i_1} \prod_{k=0}^{h_l^{-1}} G_k^l(z_{kh_l}(l))-\sum_{i_2=1}^d\int_{0}^r\alpha(Z_s)_{i_1i_2}dW_{s}(i_2)
\exp\{-\frac{1}{\gamma}(\phi(Z_1)+h_l\sum_{k=1}^{h_l^{-1}-1} \ell(Z_{kh_l}))\}\Big)^2\Big] \leq
$$
$$
2\mathbb{E}_D\Big[\Big(\Big(\varphi_l(z_{h_l:M_lh_l}(l))_{i_1}-\sum_{i_2=1}^d\int_{0}^r\alpha(Z_s)_{i_1i_2}dW_{s}(i_2)\Big)\Big(\prod_{k=0}^{h_l^{-1}} G_k^l(z_{kh_l}(l))\Big)\Big)^2\Big] +
$$
$$
2\mathbb{E}_D\Big[\Big(\sum_{i_2=1}^d\int_{0}^r\alpha(Z_s)_{i_1i_2}dW_{s}(i_2)\Big)^2\Big(\prod_{k=0}^{h_l^{-1}} G_k^l(z_{kh_l}(l))-\exp\{-\frac{1}{\gamma}(\phi(Z_1)+h_l\sum_{k=1}^{h_l^{-1}-1} \ell(Z_{kh_l}))\}\Big)^2\Big].
$$
Application of Lemmata \ref{lem:a2}-\ref{lem:a3} yields the upper-bound of $Ch_l$. For the second term on the R.H.S.~of \eqref{eq:3}, one can use Lemma \ref{lem:a4}.
Hence the proof is completed.
\end{proof}

\begin{lem}\label{lem:a1}
Assume (A\ref{ass:1}). Then there exists a $0<\underline{C}<\overline{C}<+\infty$ such that for any $l\in\{M,\dots,L\}$, $z_{h_l:1}\in\mathbb{R}^{d2^l}$
$$
\underline{C} \leq \prod_{k=0}^{h_l^{-1}}G_k^l(z_{kh_l}) \leq \overline{C}
$$
and for any $l\in\{M+1,\dots,L\}$, $(z_{h_l:1}(l),z_{h_{l-1}:1}(l))\in\mathbb{R}^{d(2^l+2^{l-1})}$
$$
\underline{C} \leq \Big\{\prod_{k\in K_l^1}
\check{G}_k^l(z_{kh_l}(l),z_{a_k(l)}(l-1))\Big\}
\Big\{\prod_{k\in K_l^2}\check{G}_k^l(z_{kh_l}(l))\Big\} \leq \overline{C}.
$$
\end{lem}

\begin{proof}
Throughout $0<\underline{C} <\overline{C}<+\infty$ are finite constants that do not depend upon $l$ and whose value may change on each appearance.
We note that for any $k\in\{0,\dots,h_l^{-1}-1\}$, any $l\in\{M,\dots,L\}$ and any $z\in\mathbb{R}^d$
$$
\exp\{-\frac{h_l}{\gamma}\sup_z|\ell(z)|\} \leq G_k^l(z) \leq \exp\{-\frac{h_l}{\gamma}\inf_z\ell(z)\}.
$$
Clearly for any $z\in\mathbb{R}^d$ $\underline{C} \leq G_{h_l^{-1}}^l(z)\leq \overline{C}$, hence it follows that for any $l\in\{M,\dots,L\}$, $z_{h_l:1}\in\mathbb{R}^{d2^l}$
$$
\underline{C} \leq \prod_{k=0}^{h_l^{-1}}G_k^l(z_{kh_l}) \leq \overline{C}.
$$
The second result is established using the relationship between $\check{G}_k^l$ and $G_k^l$.
\end{proof}

Recall we are denoting expectations w.r.t.~the law of the diffusion \eqref{eq:diff} as $\mathbb{E}_{D}$ and $W_{s}(i_2)$ is the $i_2^{th}-$element of
the Brownian motion in \eqref{eq:diff}.

\begin{lem}\label{lem:a2}
Assume (A\ref{ass:1}). Then for any $i_1\in\{1,\dots,m\}$ there exist a $C<+\infty$ such that for any $l\in\{M,\dots,L\}$
$$
\mathbb{E}_D\Big[\Big(\Big(\varphi_l(z_{h_l:M_lh_l}(l))_{i_1}-\sum_{i_2=1}^d\int_{0}^r\alpha(Z_s)_{i_1i_2}dW_{s}(i_2)\Big)\Big(\prod_{k=0}^{h_l^{-1}} G_k^l(z_{kh_l}(l))\Big)\Big)^2\Big] \leq C h_l.
$$
\end{lem}

\begin{proof}
By Lemma \ref{lem:a1} we need only deal with
$$
\mathbb{E}_D\Big[\Big(\varphi_l(z_{h_l:M_lh_l}(l))_{i_1}-\sum_{i_2=1}^d\int_{0}^r\alpha(Z_s)_{i_1i_2}dW_{s}(i_2)\Big)^2\Big].
$$
This latter term, via the $C_2-$inequality is upper-bounded by
$$
2\mathbb{E}_D\Big[\Big(\varphi_l(z_{h_l:M_lh_l}(l))_{i_1}-\sum_{i_2=1}^d\sum_{k=1}^{M_l}\alpha(Z_{(k-1)h_l})_{i_1i_2}W_{k}^l(i_2)\Big)^2\Big] +
$$
$$
2\mathbb{E}_D\Big[\Big(\sum_{i_2=1}^d\sum_{k=1}^{M_l}\alpha(Z_{(k-1)h_l})_{i_1i_2}W_{k}^l(i_2)-\sum_{i_2=1}^d\int_{0}^r\alpha(Z_s)_{i_1i_2}dW_{s}(i_2)\Big)^2\Big].
$$
We treat these two terms independently, calling them $T_1$ and $T_2$ respectively.

\textbf{Term}: $T_1$. By repeated use of the $C_2-$inequality, $T_1$ is upper-bounded by
$$
C\sum_{i_2=1}^d\mathbb{E}\Big[\Big(\sum_{k=1}^{M_l}(\alpha(Z_{(k-1)h_l}(l))_{i_1i_2}-\alpha(Z_{(k-1)h_l})_{i_1i_2})W_{k}^l(i_2)\Big)^2\Big] = 
$$
$$
C\sum_{i_2=1}^d\sum_{k=1}^{M_l}\mathbb{E}\Big[\Big((\alpha(Z_{(k-1)h_l}(l))_{i_1i_2}-\alpha(Z_{(k-1)h_l})_{i_1i_2})W_{k}^l(i_2)\Big)^2\Big] \leq
$$
$$
C\sum_{i_2=1}^d\sum_{k=1}^{M_l}\mathbb{E}\Big[\Big|Z_{(k-1)h_l}(l)-Z_{(k-1)h_l}\Big|^2 W_{k}^l(i_2)^2\Big]
$$
where we have used the Lipschitz property of $\alpha$ to go-to the last line. Splitting the expectation of the summand using Cauchy Schwarz we have
$$
T_1 \leq C\sum_{i_2=1}^d\sum_{k=1}^{M_l}\mathbb{E}\Big[\Big|Z_{(k-1)h_l}(l)-Z_{(k-1)h_l}\Big|^4 \Big]^{1/2}\mathbb{E}[W_{k}^l(i_2)^4]^{1/2}.
$$
Then using standard results from Euler-discretization of diffusions (see e.g.~\cite{kurtz}) and Gaussian distributions:
$$
\mathbb{E}\Big[\Big|Z_{(k-1)h_l}(l)-Z_{(k-1)h_l}\Big|^4 \Big]^{1/2} = \mathcal{O}(h_l) \quad\textrm{and}\quad \mathbb{E}[W_{k}^l(i_2)^4]^{1/2} = \mathcal{O}(h_l)
$$
hence
$$
T_1 \leq C\sum_{i_2=1}^d M_l h_l^2 \leq C h_l.
$$

\textbf{Term}: $T_2$. By repeated use of the $C_2-$inequality, $T_2$ is upper-bounded by
$$
C\sum_{i_2=1}^d\mathbb{E}_D\Big[\Big(\sum_{k=1}^{M_l}\int_{(k-1)h_l}^{kh_l}\{\alpha(Z_{(k-1)h_l})_{i_1i_2}-\alpha(Z_s)_{i_1i_2}\}dW_{s}(i_2)\Big)^2\Big]  = 
$$
$$
C\sum_{i_2=1}^d\sum_{k=1}^{M_l}\mathbb{E}_D\Big[\Big(\int_{(k-1)h_l}^{kh_l}\{\alpha(Z_{(k-1)h_l})_{i_1i_2}-\alpha(Z_s)_{i_1i_2}\}dW_{s}(i_2)\Big)^2\Big].
$$
Using the Ito isometry formula, clearly
$$
T_2 \leq C\sum_{i_2=1}^d\sum_{k=1}^{M_l}\int_{(k-1)h_l}^{kh_l}\mathbb{E}_D\Big[\Big(\alpha(Z_{(k-1)h_l})_{i_1i_2}-\alpha(Z_s)_{i_1i_2}\Big)^2\Big]ds
$$
Then using the Lipschitz property of $\alpha$ and then standard results for diffusion processes
$$
T_2 \leq C h_l.
$$
The proof is hence completed by the above arguments.
\end{proof}

\begin{lem}\label{lem:a3}
Assume (A\ref{ass:1}). Then for any $i_1\in\{1,\dots,m\}$ there exist a $C<+\infty$ such that for any $l\in\{M,\dots,L\}$
$$
\mathbb{E}_D\Big[\Big(\sum_{i_2=1}^d\int_{0}^r\alpha(Z_s)_{i_1i_2}dW_{s}(i_2)\Big)^2\Big(\prod_{k=0}^{h_l^{-1}} G_k^l(z_{kh_l}(l))-\exp\{-\frac{1}{\gamma}(\phi(Z_1)+h_l\sum_{k=1}^{h_l^{-1}-1} \ell(Z_{kh_l}))\}\Big)^2\Big] \leq C h_l.
$$
\end{lem}

\begin{proof}
We begin by splitting the expectation via Cauchy Schwarz to yield the upper-bound
\begin{equation}\label{eq:2}
\mathbb{E}_D\Big[\Big(\sum_{i_2=1}^d\int_{0}^r\alpha(Z_s)_{i_1i_2}dW_{s}(i_2)\Big)^4\Big]^{1/2}\mathbb{E}_D\Big[\Big(\prod_{k=0}^{h_l^{-1}} G_k^l(z_{kh_l}(l))-\exp\{-\frac{1}{\gamma}(\phi(Z_1)+h_l\sum_{k=1}^{h_l^{-1}-1} \ell(Z_{kh_l}))\}\Big)^2\Big]^{1/2}.
\end{equation}
By (a classic variation of) the Burkholder-Gundy-Davis inequality, it follows that
$$
\mathbb{E}_D\Big[\Big(\sum_{i_2=1}^d\int_{0}^r\alpha(Z_s)_{i_1i_2}dW_{s}(i_2)\Big)^4\Big]^{1/2} \leq C
$$
hence we consider the right-hand expectation in \eqref{eq:2}. Using the Lipschitz property of $e^{-x}$ for bounded $x$, it follows that
$$
\mathbb{E}_D\Big[\Big(\prod_{k=0}^{h_l^{-1}} G_k^l(z_{kh_l}(l))-\exp\{-\frac{1}{\gamma}(\phi(Z_1)+h_l\sum_{k=1}^{h_l^{-1}-1} \ell(Z_{kh_l}))\}\Big)^2\Big]^{1/2} \leq
$$
$$
C\mathbb{E}_D\Big[\Big(\phi(Z_1(l))-\phi(Z_1) + h_l\sum_{k=1}^{h_l^{-1}-1}(\ell(Z_{kh_l}(l))-\ell(Z_{kh_l}))\Big)^4\Big]^{1/2}. 
$$
Using the Lipschitz property of $\phi$ and $\ell$ it clearly follows that
$$
\mathbb{E}_D\Big[\Big(\phi(Z_1(l))-\phi(Z_1) + h_l\sum_{k=1}^{h_l^{-1}-1}(\ell(Z_{kh_l}(l))-\ell(Z_{kh_l}))\Big)^4\Big]^{1/2}
\leq 
\mathbb{E}_D\Big[\max_{k\in\{1,\dots,h_l^{-1}\}}|Z_{kh_l}(l)-Z_{kh_l}|^4\Big]^{1/2}.
$$
Hence using standard results for Euler approximations of diffusion processes
$$
\mathbb{E}_D\Big[\Big(\prod_{k=0}^{h_l^{-1}} G_k^l(z_{kh_l}(l))-\exp\{-\frac{1}{\gamma}(\phi(Z_1)+h_l\sum_{k=1}^{h_l^{-1}-1} \ell(Z_{kh_l}))\}\Big)^2\Big]^{1/2} \leq C h_l
$$
and one can conclude the proof by the above arguments.
\end{proof}

\begin{lem}\label{lem:a4}
Assume (A\ref{ass:1}). Then for any $i_1\in\{1,\dots,m\}$ there exist a $C<+\infty$ such that for any $l\in\{M,\dots,L\}$
$$
\mathbb{E}_D\Big[\Big(\Big(\sum_{i_2=1}^d\int_{0}^r\alpha(Z_s)_{i_1i_2}dW_{s}(i_2)\Big)\exp\{-\frac{1}{\gamma}(\phi(Z_1)+h_l\sum_{k=1}^{h_l^{-1}-1} \ell(Z_{kh_l}))\}
$$
$$
-\exp\{-\frac{1}{\gamma}(\phi(Z_1)+\int_{0}^1 \ell(Z_s)ds)\}
\Big)^2\Big] \leq C h_l.
$$
\end{lem}

\begin{proof}
As in the proof of Lemma \ref{lem:a3} split the expectation via Cauchy Schwarz to yield the upper-bound
$$
\mathbb{E}_D\Big[\Big(\sum_{i_2=1}^d\int_{0}^r\alpha(Z_s)_{i_1i_2}dW_{s}(i_2)\Big)^4\Big]^{1/2}
\mathbb{E}_D\Big[\Big(\exp\{-\frac{1}{\gamma}(\phi(Z_1)+h_l\sum_{k=1}^{h_l^{-1}-1} \ell(Z_{kh_l}))\}
-\exp\{-\frac{1}{\gamma}(\phi(Z_1)+\int_{0}^1 \ell(Z_s)ds)\}
\Big)^4\Big]^{1/2}.
$$
Again (as in the proof of Lemma \ref{lem:a3}) by the Burkholder-Gundy-Davis inequality the first expectation is $\mathcal{O}(1)$ so we concentrate upon the second expectation (call it $T_1$). Using the Lipschitz property of $e^{-x}$ for bounded $x$, it follows that
\begin{eqnarray*}
T_1 & \leq &  C \mathbb{E}_D\Big[\Big(h_l\sum_{k=1}^{h_l^{-1}-1} \ell(Z_{kh_l}))-\int_{0}^1 \ell(Z_s)ds\Big)^4\Big]^{1/2}\\\
& = & C \mathbb{E}_D\Big[\Big(\sum_{k=1}^{h_l^{-1}-1}\int_{(k-1)h_l}^{kh_l}\{\ell(Z_{kh_l}))-\ell(Z_s)\}ds + \int_{(1-h_l)}^1\ell(Z_s)ds\Big)^4\Big]^{1/2}.
\end{eqnarray*}
Then via the Minkowski inequality
$$
T_1  \leq C\Big[\sum_{k=1}^{h_l^{-1}-1}\mathbb{E}_D\Big[\Big(\int_{(k-1)h_l}^{kh_l}\{\ell(Z_{kh_l}))-\ell(Z_s)\}ds\Big)^4\Big]^{1/4}+ Ch_{l}\Big]^2
$$
where we have used the fact that $\ell$ is bounded. Then applying Jensen with standard results in diffusions, yields 
$
T_1  \leq C h_l.
$
\end{proof} 

\subsection{Proof of Proposition \ref{prop:main_prop}}\label{app:prop_proof}

\begin{proof}[Proof of Proposition \ref{prop:main_prop}]
We make the defintions:
\begin{eqnarray*}
a^N & := & \frac{1}{N_l}\sum_{i=1}^{N_l}\varphi_l(z_{h_l:M_lh_l}^i(l))
\check{H}^{l,1}(z_{h_l:1}^i(l),z_{h_{l-1}:1}^i(l-1)) \\
b^N & := & \frac{1}{N_l}\sum_{i=1}^{N_l}
\check{H}^{l,1}(z_{h_l:1}^i(l),z_{h_{l-1}:1}^i(l-1)) \\
c^N & := & \frac{1}{N_l}\sum_{i=1}^{N_l}\varphi_{l-1}(z_{h_{l-1}:M_{l-1}h_{l-1}}^i(l-1))
\check{H}^{l,2}(z_{h_l:1}^i(l),z_{h_{l-1}:1}^i(l-1)) \\
d^N & := & \frac{1}{N_l}\sum_{i=1}^{N_l}
\check{H}^{l,2}(z_{h_l:1}^i(l),z_{h_{l-1}:1}^i(l-1)) \\
a & := & \mathbb{E}_{\pi^{l,l-1}}[\varphi_l(Z_{h_l:M_lh_l}(l))
\check{H}^{l,1}(Z_{h_l:1}(l),Z_{h_{l-1}:1}(l-1))] \\
b & := & \mathbb{E}_{\pi^{l,l-1}}[
\check{H}^{l,1}(Z_{h_l:1}(l),Z_{h_{l-1}:1}(l-1))] \\
c & := & \mathbb{E}_{\pi^{l,l-1}}[\varphi_{l-1}(Z_{h_{l-1}:M_{l-1}h_{l-1}}(l-1))
\check{H}^{l,2}(Z_{h_l:1}(l),Z_{h_{l-1}:1}(l-1))]\\
d & := & \mathbb{E}_{\pi^{l,l-1}}[
\check{H}^{l,2}(Z_{h_l:1}(l),Z_{h_{l-1}:1}(l-1))].
\end{eqnarray*}
Then component-wise (from here, our calculations should be considered component-by-component):
$$
r\Big(u^{l,N_{l}}(0,x_0)-u^{l-1,N_{l}}(0,x_0)-\{u^{l}(0,x_0)-u^{l-1}(0,x_0)\}\Big) = \frac{a^N}{b^N} - \frac{c^N}{d^N} - \Big(\frac{a}{b} - \frac{c}{d}\Big).
$$
By \cite[Lemma D.5]{mlpf}
$$
\frac{a^N}{b^N} - \frac{c^N}{d^N} - \Big(\frac{a}{b} - \frac{c}{d}\Big) = 
$$
$$
\frac{1}{b^N}(a^N-c^N-(a-c)) - \frac{c^N}{b^Nd^N}(b^N-d^N-(b-d)) + \frac{1}{b^Nb}(b-b^N)(a-c) +
$$
$$
\frac{c}{d^Nbd}(d^N-d)(b-d) - \frac{1}{b^Nd^N}(c^N-c)(b-d) + \frac{c}{b^Nd^Nb}(b^N-b)(b-d).
$$
The first two terms on the R.H.S.~and the last four terms on the R.H.S.~can be treated using similar calculations, so we only consider
$$
T_1:=\frac{1}{b^N}(a^N-c^N-(a-c)) \quad\textrm{and}\quad T_2:=\frac{1}{b^Nb}(b-b^N)(a-c).
$$

\textbf{Term}: $T_1$. We have 
$$
T_1 = \Big(\frac{1}{b^N}-\frac{1}{b}\Big)\Big(a^N-c^N-(a-c)\Big) + \frac{1}{b}(a^N-c^N-(a-c)).
$$
On taking expectations w.r.t.~the law of the simulated Markov chain, we have
$$
\mathbb{E}[T_1] = \mathbb{E}\Big[\Big(\frac{1}{b^N}-\frac{1}{b}\Big)\Big(a^N-c^N-(a-c)\Big)\Big]
$$
as the chain is started in stationarity. Applying the Cauchy-Schwarz inequality:
$$
|\mathbb{E}[T_1]|\leq \mathbb{E}\Big[\Big(\frac{1}{b^N}-\frac{1}{b}\Big)^2\Big]^{1/2}\mathbb{E}[(a^N-c^N-(a-c))^2]^{1/2}.
$$
By using a similar result to \cite[Proposition A.1]{jasra} and by Proposition \ref{prop:a1}
$$
\mathbb{E}[(a^N-c^N-(a-c))^2]^{1/2} \leq \frac{Ch_l^{1/2}}{N_l^{1/2}}.
$$
By standard results for uniformly ergodic Markov chains (e.g.~the result in \cite[Proposition A.1]{jasra}), along with Lemma \ref{lem:a1}
$$
\mathbb{E}\Big[\Big(\frac{1}{b^N}-\frac{1}{b}\Big)^2\Big]^{1/2} \leq \frac{C}{N_l^{1/2}}
$$
and hence
$$
|\mathbb{E}[T_1]|\leq \frac{Ch_l^{1/2}}{N_l}.
$$

\textbf{Term}: $T_2$. We have
$$
T_2 = \Big(\frac{1}{b^Nb}-\frac{1}{b^2}\Big)(b-b^N)(a-c) + \frac{1}{b^2}(b-b^N)(a-c)
$$
and taking expectations as for $T_1$
$$
\mathbb{E}[T_2] = \mathbb{E}\Big[\Big(\frac{1}{b^Nb}-\frac{1}{b^2}\Big)(b-b^N)(a-c)\Big].
$$
Hence, by Cauchy-Schwarz
$$
|\mathbb{E}[T_2]| \leq |a-c|\mathbb{E}\Big[\Big(\frac{1}{b^Nb}-\frac{1}{b^2}\Big)^2\Big]^{1/2}\mathbb{E}[(b-b^N)^2]^{1/2}.
$$
By Proposition \ref{prop:a1} $|a-c|\leq C h_l^{1/2}$ and again, standard results for uniformly ergodic Markov chains, along with Lemma \ref{lem:a1}
$$
\mathbb{E}\Big[\Big(\frac{1}{b^Nb}-\frac{1}{b^2}\Big)^2\Big]^{1/2}\mathbb{E}[(b-b^N)^2]^{1/2} \leq \frac{C}{N_l}
$$
and hence
$$
|\mathbb{E}[T_2]|\leq \frac{Ch_l^{1/2}}{N_l}.
$$
\end{proof}

\end{document}